\documentclass[12pt]{article}

\usepackage[T1]{fontenc}
\usepackage[english]{babel}
\usepackage{mathtools}
\usepackage{amssymb}
\usepackage{amsthm}
\usepackage{graphicx}
\usepackage{mathrsfs}
\usepackage{microtype}
\usepackage[margin=1in]{geometry}
\usepackage[doublespacing]{setspace}
\usepackage{xcolor}
\usepackage{tikz}
\usetikzlibrary{arrows.meta,positioning}
\usepackage[authoryear,round]{natbib}
\usepackage[colorlinks=true,linkcolor=blue,citecolor=blue,urlcolor=blue]{hyperref}
\numberwithin{equation}{section}
\allowdisplaybreaks[3]

\AtBeginDocument{%
  \setlength{\abovedisplayskip}{5pt plus 2pt minus 2pt}%
  \setlength{\belowdisplayskip}{5pt plus 2pt minus 2pt}%
  \setlength{\abovedisplayshortskip}{3pt plus 1pt minus 1pt}%
  \setlength{\belowdisplayshortskip}{4pt plus 1pt minus 1pt}%
  \setlength{\jot}{2pt}%
  \setlength{\arraycolsep}{3pt}%
  \setlength{\multlinegap}{8pt}%
}

\theoremstyle{plain}
\newtheorem{theo}{Theorem}[section]
\newtheorem{prop}[theo]{Proposition}
\newtheorem{lemm}[theo]{Lemma}
\newtheorem{coro}[theo]{Corollary}

\theoremstyle{definition}
\newtheorem{defi}[theo]{Definition}

\newtheorem{assu}[theo]{Assumption}
\newtheorem{mainassu}{Assumption}
\theoremstyle{remark}
\newtheorem{rema}[theo]{Remark}

\DeclareMathOperator{\KL}{KL}
\DeclareMathOperator{\Tr}{Tr}
\DeclareMathOperator{\Cov}{Cov}

\newcommand{\R}{\mathbb R}
\newcommand{\E}{\mathbb E}

\newcommand{\bbeta}{{\boldsymbol{\beta}}}
\newcommand{\cX}{\mathcal X}
\newcommand{\Pbar}{\overline P}
\newcommand{\Rbar}{\overline R}

\begin{document}

\title{Forward-Evolution Error Analysis and Adaptive Design for Matrix-Valued Diffusion Models}

\author{%
Tongyao Pang\textsuperscript{1}\thanks{Corresponding author. Email: \href{mailto:typang@tsinghua.edu.cn}{typang@tsinghua.edu.cn}.}\quad
Zuowei Shen\textsuperscript{2}\thanks{Email: \href{mailto:matzuows@nus.edu.sg}{matzuows@nus.edu.sg}.}\quad
Ruitong Zhang\textsuperscript{3}\thanks{Email: \href{mailto:zrt22@mails.tsinghua.edu.cn}{zrt22@mails.tsinghua.edu.cn}.}\\[0.25em]
{\small \textsuperscript{1}Yau Mathematical Sciences Center, Tsinghua University}\\[-0.1em]
{\small \textsuperscript{2}Department of Mathematics, National University of Singapore}\\[-0.1em]
{\small \textsuperscript{3}Department of Mathematical Sciences, Tsinghua University}}
\date{}

\renewcommand{\thefootnote}{\fnsymbol{footnote}}
\maketitle
\renewcommand{\thefootnote}{\arabic{footnote}}
\setcounter{footnote}{0}

\begin{abstract}
Diffusion models learn to reverse a predefined corruption process, but sampling
still requires a costly time discretization and depends on the chosen noise
schedule. We study these two issues for variance-preserving diffusions with
matrix-valued schedules. Our analysis transfers reverse-time discretization
errors to the forward corruption law and treats two numerical schemes within a
common framework. The first freezes the score and yields, through a
matrix-sensitive local comparison and forward information dissipation, an
ambient-dimensional step complexity with leading factor
\(d/\varepsilon^2\) for KL accuracy \(\varepsilon^2\). The second keeps the
known Gaussian drift exact and freezes the posterior mean. For data of
metric-entropy dimension \(k\), a forward Markov identity, an anisotropic
covering estimate, and Stieltjes integration by parts give the corresponding
factor \(k\log k/\varepsilon^2\). In both cases, the proof identifies a local
error, accumulates it through the forward evolution, and inserts the result
into a common KL decomposition. The local errors further provide directional
criteria for matrix schedules and an asymptotically optimal square-root
adaptive grid. A high-dimensional Gaussian-mixture experiment illustrates the
resulting schedule and grid improvements.
\end{abstract}


\section{Introduction}

Generative models have become a central tool for learning and sampling complex
high-dimensional distributions. Their progress is especially visible in image
generation, where modern models synthesize diverse, high-resolution images and
support large-scale text-conditioned generation
\cite{rombach2022latent,blackforestlabs2024flux}. Autoregressive models directly
factorize the data distribution and generate one variable at a time
\cite{oord2016pixel}. Many other prominent approaches instead begin with a
simple latent or reference distribution and learn a mechanism that carries it
toward the data law. Variational autoencoders learn a latent-variable model,
normalizing flows learn invertible transports in discrete or continuous time,
and generative adversarial networks learn through a two-player objective
\cite{kingma2014autoencoding,dinh2017realnvp,chen2018neuralode,
goodfellow2014generative}.
Diffusion models have emerged as another leading instance of this transport
view and have enabled high-quality image generation at scale
\cite{sohl2015deep,ho2020denoising,song2020score}.

Diffusion models differ structurally from these alternatives because the map
from data toward the reference law is prescribed rather than learned. A forward
process gradually corrupts data into noise, and the generative model only has
to learn how to reverse this known evolution. For Gaussian corruptions, a
training pair at any noise level can be sampled directly from a clean datum and
independent Gaussian noise. Denoising score matching then reduces learning the
reverse drift to a supervised conditional-prediction problem
\cite{vincent2011connection,song2019generative}. Training therefore requires no
simulation of the reverse dynamics and is decoupled from the numerical sampler
used at generation time. This predefined, simulation-free training mechanism
is one reason diffusion objectives are simple to implement and scale.

The same separation leaves two important design problems at sampling time.
First, reversing the corruption is sequential and usually requires many model
evaluations, so the reverse-time discretization controls generation cost.
Second, the forward process itself must be chosen. Its noise schedule determines
the intermediate distributions seen in training, the rate of terminal
Gaussianization, and the difficulty of the reverse dynamics. Scalar schedules
control only the global noise level and time parameterization. Matrix-valued
schedules can additionally allocate contraction and noise across directions.
This motivates us to study sampling complexity together with the design of the
matrix schedule and the reverse-time grid.

Let \(q_0\) be the data law on \(\R^d\). We consider the variance-preserving Ornstein--Uhlenbeck diffusion
\begin{equation}\label{eq:intro-forward-SDE}
dX_u=-\frac12\bbeta(u)X_u\,du+\bbeta(u)^{1/2}dW_u,
\qquad X_0\sim q_0,
\end{equation}
where \(u\) is forward time and \(\bbeta(u)\) is symmetric positive definite.
The standard Gaussian law \(\pi=N(0,I)\) is invariant. A scalar schedule changes
only the speed of the OU flow. A matrix schedule also changes its directional
geometry, and a noncommuting schedule cannot be reduced to a scalar time
change. This additional freedom creates a coupled design problem because the
schedule affects terminal mixing, prediction geometry, and discretization.

The convergence theory of diffusion samplers separates prediction
approximation, time discretization, and terminal initialization. Existing
results provide polynomial guarantees under increasingly weak assumptions and,
for the isotropic OU process, nearly linear dependence on the ambient dimension
\cite{lee2022convergence,chen2022sampling,chen2023improved,benton2023nearly}.
The sharp ambient-dimensional analysis in \cite{benton2023nearly} uses
stochastic localization to control posterior covariance along the corruption
path. Related work obtains intrinsic-dimensional rates for subspace, manifold,
and metric-entropy models
\cite{chen2023lowdimensional,tang2024adaptivity,liyan2024adapting,huang2026ddpm}.
These results largely fix the forward process and focus on the resulting
sampler. Our starting point is instead the structural advantage that makes
diffusion training simple: the forward evolution is known and directly
simulable. We transfer the reverse-time discretization error to the forward
corruption law, where it can be accumulated using forward identities. For score
freezing this gives a concise alternative to a localization argument, and the
resulting forward-marginal functionals also reveal how the schedule and grid
enter the error.

The prediction target changes the numerical scheme. At a fixed noise
level, the score and the posterior mean
\(m_u(x)=\E[X_0\mid X_u=x]\), the minimum mean-square error (MMSE) estimator of
\(X_0\), are related by Tweedie's identity. Freezing the full score, however,
is not the same as freezing only \(m_u\) while retaining the known
time-dependent Gaussian drift. The second scheme is particularly natural for
low-dimensional data because the clean prediction remains tied to the data
geometry, and it is consistent with empirical evidence favoring clean-image
prediction \cite{li2026back}. We analyze both schemes. Neither is claimed to
dominate uniformly: score freezing gives a canonical ambient-dimensional
baseline under weak assumptions, while posterior-mean freezing exposes
intrinsic-dimensional adaptation under additional geometric assumptions.

\subsection{Main contributions}

We study two reverse-time discretizations: one freezes the full score, while
the other integrates the Gaussian linear drift and freezes the posterior mean.
Let \(u_t:=T-t\) and \(u_i:=T-t_i\). The matrices \(L_u,D_u\), defined in
Section~2, give the posterior-mean representation of the exact reverse drift.
Figure~\ref{fig:error-decomposition-processes} displays the two approximation
and discretization paths.
\begin{samepage}
Their output laws, denoted by
\(p_{T-t_{\min}}\) for score freezing and \(\widehat p_{T-t_{\min}}\) for
posterior-mean freezing, satisfy
\begin{equation}\label{eq:intro-error-decomposition}
\begin{aligned}
\KL(q_{t_{\min}}\|p_{T-t_{\min}})
&\leq
\underbrace{\epsilon_{\mathrm{score}}^2}_{\text{score approximation}}
+\underbrace{\operatorname{Disc}_{\mathrm{score}}(\mathcal T_N)}
_{\text{score-freezing discretization}}
+\underbrace{\KL(q_T\|\pi)}_{\text{terminal initialization}},\\
\KL(q_{t_{\min}}\|\widehat p_{T-t_{\min}})
&\leq
\underbrace{\epsilon_{\mathrm{pm}}^2}_{\text{posterior-mean approximation}}
+\underbrace{\operatorname{Disc}_{\mathrm{pm}}(\mathcal T_N)}
_{\text{posterior-mean-freezing discretization}}
+\underbrace{\KL(q_T\|\pi)}_{\text{terminal initialization}}.
\end{aligned}
\end{equation}
\end{samepage}

\begin{figure}[!t]
\centering
\begin{tikzpicture}[
  wide/.style={
    draw=black!45,
    fill=black!1,
    text width=0.86\textwidth,
    inner xsep=6pt,
    inner ysep=4pt,
    align=left,
    font=\footnotesize
  },
  branch/.style={
    wide,
    text width=0.405\textwidth,
    font=\scriptsize
  },
  init/.style={
    draw=black!35,
    fill=black!1,
    text width=0.405\textwidth,
    inner xsep=4pt,
    inner ysep=3pt,
    align=center,
    font=\scriptsize
  },
  relation/.style={
    -{Stealth[length=2mm]},
    draw=cyan!55!black,
    line width=0.85pt
  },
  tag/.style={
    font=\scriptsize,
    text=black,
    fill=white,
    inner sep=1pt
  }
]
\node[wide] (original) {\(
  (\mathsf O)\quad
  dX_u=-\tfrac12\bbeta(u)X_u\,du+\bbeta(u)^{1/2}dW_u,
  \qquad X_0\sim q_0.
\)};
\node[wide, below=6mm of original] (reverse) {\(
  \begin{aligned}
  (\mathsf R)\quad dY_t
  &=\{\tfrac12\bbeta(u_t)Y_t+\bbeta(u_t)h_{u_t}(Y_t)\}\,dt
    +\bbeta(u_t)^{1/2}dB_t\\
  &=\{L_{u_t}Y_t+D_{u_t}m_{u_t}(Y_t)\}\,dt
    +\bbeta(u_t)^{1/2}dB_t,
  \qquad Y_0\sim q_T.
  \end{aligned}
\)};

\node[branch, anchor=north]
  (trained-score) at ([xshift=-0.22\textwidth,yshift=-9mm]reverse.south) {\(
  \begin{aligned}
  (\mathsf R_{\rm sc})\quad dY_t^{\rm sc}
  &=\{\tfrac12\bbeta(u_t)Y_t^{\rm sc}
    +\bbeta(u_t)s_\theta(Y_t^{\rm sc},u_t)\}\,dt\\
  &\quad+\bbeta(u_t)^{1/2}dB_t,
  \qquad Y_0^{\rm sc}\sim q_T.
  \end{aligned}
\)};
\node[branch, anchor=north]
  (trained-pm) at ([xshift=0.22\textwidth,yshift=-9mm]reverse.south) {\(
  \begin{aligned}
  (\mathsf R_{\rm pm})\quad dY_t^{\rm pm}
  &=\{L_{u_t}Y_t^{\rm pm}
    +D_{u_t}\widehat m_{u_t}(Y_t^{\rm pm})\}\,dt\\
  &\quad+\bbeta(u_t)^{1/2}dB_t,
  \qquad Y_0^{\rm pm}\sim q_T.
  \end{aligned}
\)};

\node[branch, below=10mm of trained-score] (discrete-score) {\(
  \begin{aligned}
  (\mathsf D_{\rm sc}^{q_T})\quad d\widehat Y_t^{\rm sc}
  &=\{\tfrac12\bbeta(u_t)\widehat Y_t^{\rm sc}
    +\bbeta(u_t)s_\theta(\widehat Y_{t_i}^{\rm sc},u_i)\}\,dt\\
  &\quad+\bbeta(u_t)^{1/2}d\widetilde B_t,
  \qquad \widehat Y_0^{\rm sc}\sim q_T.
  \end{aligned}
\)};
\node[branch, below=10mm of trained-pm] (discrete-pm) {\(
  \begin{aligned}
  (\mathsf D_{\rm pm}^{q_T})\quad d\widehat Y_t^{\rm pm}
  &=\{L_{u_t}\widehat Y_t^{\rm pm}
    +D_{u_t}\widehat m_{u_i}(\widehat Y_{t_i}^{\rm pm})\}\,dt\\
  &\quad+\bbeta(u_t)^{1/2}d\widetilde B_t,
  \qquad \widehat Y_0^{\rm pm}\sim q_T.
  \end{aligned}
\)};

\node[init, below=9mm of discrete-score] (init-score) {\(
  (\mathsf D_{\rm sc}^{\pi}):\ \text{same dynamics as }
  (\mathsf D_{\rm sc}^{q_T}),\quad \widehat Y_0^{\rm sc}\sim\pi.
\)};
\node[init, below=9mm of discrete-pm] (init-pm) {\(
  (\mathsf D_{\rm pm}^{\pi}):\ \text{same dynamics as }
  (\mathsf D_{\rm pm}^{q_T}),\quad \widehat Y_0^{\rm pm}\sim\pi.
\)};

\draw[relation] (original.south) --
  node[midway,right=1mm,tag]{time reversal} (reverse.north);
\draw[relation] (reverse.south) --
  node[pos=0.55,above left=-1pt,tag]{score approximation}
  (trained-score.north);
\draw[relation] (reverse.south) --
  node[pos=0.55,above right=-1pt,tag]{posterior-mean approximation}
  (trained-pm.north);
\draw[relation] (trained-score.south) --
  node[midway,right=1mm,tag]{score freezing} (discrete-score.north);
\draw[relation] (trained-pm.south) --
  node[midway,right=1mm,tag]{posterior-mean freezing} (discrete-pm.north);
\draw[relation] (discrete-score.south) --
  node[midway,right=1mm,tag]{initialization} (init-score.north);
\draw[relation] (discrete-pm.south) --
  node[midway,right=1mm,tag]{initialization} (init-pm.north);
\end{tikzpicture}
\caption{Error decomposition for the score and posterior-mean branches.
Time reversal gives the exact reverse process. Prediction approximation changes
the drift target, freezing produces the time-discretized process, and replacing
the exact terminal law \(q_T\) by \(\pi\) adds initialization error.}
\label{fig:error-decomposition-processes}
\end{figure}
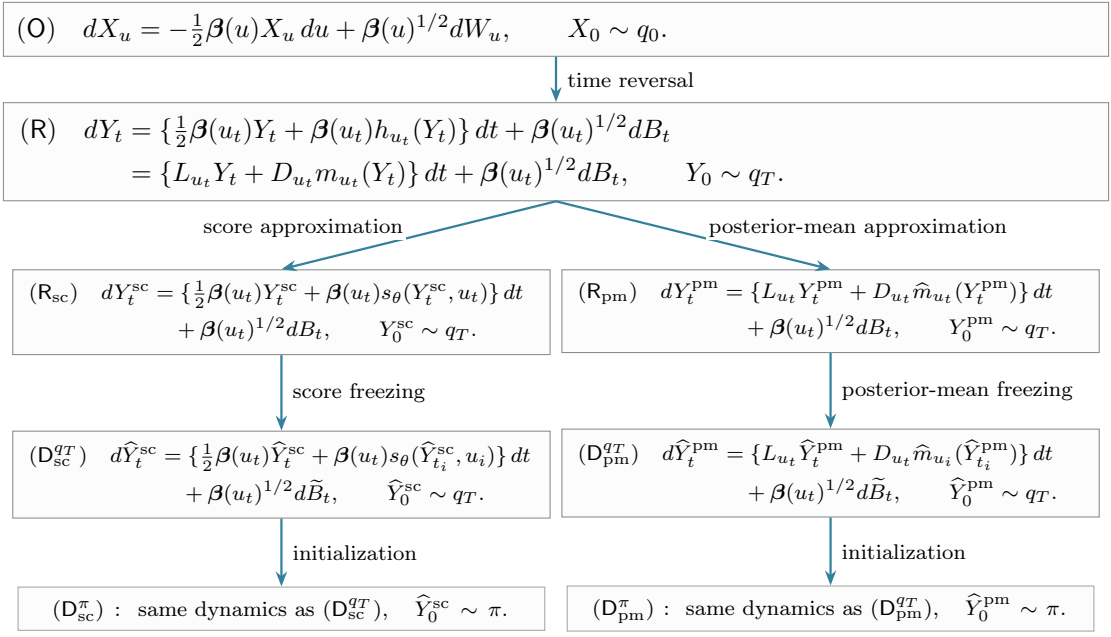
The approximation terms in \eqref{eq:intro-error-decomposition} are learning
budgets, and the initialization term is common to both branches. Our treatment
of the two discretization terms follows the same three-step architecture.

\paragraph{A unified forward analysis.}
First, we identify the local error created by freezing the prediction target.
For score freezing, a conditional-score identity gives a matrix-sensitive
comparison involving a one-time weighted score energy and the schedule
transition matrix. For posterior-mean freezing, the forward Markov property
turns the local error into an increment of clean-prediction covariance. Second,
we accumulate these local errors on a relative grid. Both arguments integrate a
positive Stieltjes measure by parts. The score branch uses weighted
Fisher-information dissipation, while the posterior-mean branch combines
covariance increments with a metric-entropy bound. These different local
quantities lead respectively to ambient- and intrinsic-dimensional dependence.
Third, we insert the cumulative estimates into the common KL decomposition
\eqref{eq:intro-error-decomposition}. The proof therefore remains on the
forward corruption law until the final comparison with the generated
distribution.

The resulting step complexities have a simple dimension--accuracy form. Let
\(d\) be the ambient dimension, let \(k\) be the metric-entropy dimension, and
let \(\varepsilon^2\) be the target KL accuracy. Under the assumptions and
parameter choices of Theorems~\ref{theo:relative-grid-kl-complexity}
and~\ref{theo:ld-main}, and apart from the logarithmic time-truncation factors
stated there, score freezing requires order \(d/\varepsilon^2\) reverse steps,
whereas posterior-mean freezing requires order
\(k\log k/\varepsilon^2\). Thus the first scheme gives the nearly linear
ambient-dimensional order, while the second adapts the step complexity to
intrinsic dimension.

\paragraph{Schedule and grid design.}
The local analysis also supplies design criteria. Under a fixed instantaneous
noise spectrum, the score-frozen growth becomes a weighted score-Hessian
energy and gives a directional rule for matrix schedules. The posterior-mean
analysis gives a complementary subspace-alignment rule. Once the schedule is
fixed, equidistributing the square root of the local growth rate produces an
asymptotically optimal grid for the leading discretization cost. These criteria
may be evaluated analytically or estimated from forward-noised pilot samples,
without reverse simulation.

\paragraph{Scope and outlook.}
Section~6 tests these predictions on a high-dimensional Gaussian mixture whose
dominant geometry changes with the noise level. The experiment isolates the
discretization term and shows benefits from both a forward-aligned rotating
schedule and the adaptive grid over several step counts. This is a controlled
proof of concept rather than a large-scale image study. It uses accurately
evaluated mixture scores and does not test learned-predictor or pilot-estimation
error. Dense matrix schedules, pilot estimation at image scale, and the
interaction between schedule changes and model training remain open practical
questions. More broadly, the same strategy of transferring reverse errors to a
prescribed forward evolution may extend to structured stochastic interpolants,
although we do not pursue that direction here \cite{albergo2025stochastic}.

\subsection{Related work}

A central question in the theory of diffusion models is to quantify terminal
initialization, prediction approximation, and discretization under weak
assumptions. Polynomial convergence was first established for smooth targets
with dependence on a log-Sobolev constant \cite{lee2022convergence}. Subsequent
work weakened the distributional and regularity assumptions for Wasserstein,
total-variation, and KL guarantees
\cite{lee2023convergence,chen2022sampling,chen2023improved,li2024towards,conforti2025kl}.

A second line of work sharpens dimension dependence. For the standard isotropic
OU process, an \(L^2\)-accurate score yields a nearly \(d\)-linear KL guarantee
through a refined stochastic-localization treatment of reverse-time
discretization
\cite{benton2023nearly,elalaoui2022information,eldan2020taming}. Recent results
adapt to subspaces, manifolds, and metric-entropy models and obtain
intrinsic-dimensional rates under several sampler designs
\cite{chen2023lowdimensional,tang2024adaptivity,liyan2024adapting,huang2026ddpm,potaptchik2024linear}.
Our ambient result preserves the scalar worst-case order while replacing the
localization-based cumulative argument with a proof through the forward
equation and exposing matrix-schedule geometry. The posterior-mean result uses
the covering argument of \cite{huang2026ddpm} in a schedule-induced anisotropic
metric, whose Euclidean specialization recovers the standard \(k\log k\) rate.

Noise schedules, time parameterizations, and fast solvers are also algorithmic
design choices
\cite{nichol2021improved,karras2022elucidating,strasman2025noise}. A scalar
schedule is a time change of the OU flow, whereas a noncommuting matrix schedule
is not. Related probability-flow ODE solvers integrate the linear drift and
approximate weighted prediction integrals
\cite{lu2022dpmsolver,lu2025dpmsolverpp}; model statistics can also adapt their
parameterization \cite{zheng2023dpmsolverv3}. We instead study frozen-prediction
reverse-SDE discretization and use the resulting local errors to design schedule
directions and reverse-time grids.

Section~2 introduces the forward and reverse processes, the two discretization
schemes, and their common KL decomposition. Sections~3 and~4 develop the
ambient score-freezing and intrinsic posterior-mean analyses, respectively.
Section~5 first studies schedule alignment under a fixed noise budget and then
designs adaptive grids for a fixed schedule. Section~6 presents the numerical
experiment, and Section~7 concludes.

\section{Preliminaries}

This section fixes the forward model, records the equivalent score and
posterior-mean representations of its reverse process, and defines the two
frozen-prediction discretizations. Their common KL decomposition separates the
learning, discretization, and initialization terms before Sections~3 and~4
analyze the two discretization branches.

Forward time is denoted by \(u\), while reverse time is denoted by \(t\); reverse time \(t\) corresponds to the forward marginal at \(T-t\). We fix a terminal horizon \(T>0\) and an early-stopping time \(0<t_{\min}<T\). For a positive semidefinite matrix \(A\), write \(\|v\|_A^2:=v^TAv\).

\subsection{Forward process}

Let \(q_u\) denote the law of the forward process \eqref{eq:intro-forward-SDE}, let \(h_u=\nabla\log q_u\) be its score for \(u>0\), and let \(\pi=N(0,I)\). We write
\(\Sigma_u:=\E[X_uX_u^T]\) for the second-moment matrix.

\begin{mainassu}[Baseline model]\label{assu:baseline}
The data law satisfies
\[
M_2:=\E\|X_0\|^2<\infty.
\]
The schedule \(\bbeta\) is continuous and symmetric positive definite. For constants \(m_\bbeta,B_\bbeta,L_\bbeta>0\), independent of \(T\),
\[
m_\bbeta I\preceq\bbeta(u)\preceq B_\bbeta I,
\]
and \(\bbeta\) is piecewise \(C^1\) on \([t_{\min},T]\), with \(\|\bbeta'(u)\|_2\leq L_\bbeta\) on each smooth interval. When \(T\to\infty\), the same bounds hold uniformly on \([0,\infty)\).
\end{mainassu}

The matrix schedule first enters through the transition matrix
\begin{equation}\label{eq:Phi}
\partial_u\Phi(u,v)=-\frac12\bbeta(u)\Phi(u,v),
\qquad
\Phi(v,v)=I.
\end{equation}
For the transition from the data law, write
\[
A_u:=\Phi(u,0),
\qquad
\Gamma_u:=I-A_uA_u^T.
\]
The forward marginal has the explicit Gaussian representation
\begin{equation}\label{eq:forward-channel}
X_u=A_uX_0+\Gamma_u^{1/2}Z,
\qquad Z\sim N(0,I),
\end{equation}
where \(Z\) is independent of \(X_0\).
Assumption~\ref{assu:baseline} implies
\[
\|\Phi(T,0)\|_2^2
\leq \exp\left\{-\int_0^T\lambda_{\min}(\bbeta(u))\,du\right\}
\leq e^{-m_\bbeta T}.
\]
Hence \(\Phi(T,0)\to0\), the forward noise covariance converges to \(I\), and
Lemma~\ref{lem:terminal-mixing} gives \(\KL(q_T\|\pi)\to0\).
If \(0\leq v\leq u\leq T\) and all matrices \(\bbeta(r)\) commute, then
\[
\Phi(u,v)=\exp\left(-\frac12\int_v^u\bbeta(r)dr\right).
\]
In the non-commutative case, this scalar time-change representation is unavailable. The transition matrix \(\Phi(u,v)\) must therefore be kept as part of the schedule geometry.

\subsection{Reverse discretization}

For \(u>0\), define the posterior mean and the schedule matrices
\[
m_u(x):=\E[X_0\mid X_u=x],
\qquad
L_u:=\frac12\bbeta(u)-\bbeta(u)\Gamma_u^{-1},
\qquad
D_u:=\bbeta(u)\Gamma_u^{-1}A_u.
\]
For the posterior-mean parameterization, also set
\begin{equation}\label{eq:ld-S-W}
S_u:=A_u^T\Gamma_u^{-1}A_u,
\qquad
W_u:=-S_u'
=A_u^T\Gamma_u^{-1}\bbeta(u)\Gamma_u^{-1}A_u
\succeq0.
\end{equation}
Here \(S_u\) is the Gram matrix of the signal map after whitening the forward
noise, and \(W_u\) is its rate of decay. In particular,
\(S_u\preceq S_v\) whenever \(u\geq v>0\).
Differentiating the Gaussian mixture density in
\eqref{eq:forward-channel} gives the matrix Tweedie identity and its associated
Girsanov weight,
\begin{equation}\label{eq:ld-tweedie}
h_u(x)=\Gamma_u^{-1}\{A_um_u(x)-x\},
\qquad
D_u^T\bbeta(u)^{-1}D_u=W_u.
\end{equation}
The scalar formula is classical \citep{efron2011tweedie}; details for the
matrix form are given in Appendix~\ref{app:intrinsic-proofs}.
Following the time-reversal formula for diffusions
\cite{anderson1982reverse,cattiaux2023time}, the exact reverse drift has the
equivalent score and posterior-mean representations
\begin{equation}\label{eq:true-reverse}
\begin{aligned}
dY_t&=b_u(Y_t)\,dt+\bbeta(u)^{1/2}dB'_t,
\qquad u=T-t,\qquad Y_0\sim q_T,\\
b_u(y)
&:=\frac12\bbeta(u)y+\bbeta(u)h_u(y)
=L_uy+D_um_u(y).
\end{aligned}
\end{equation}
Let \(Q\) denote its path law with \(Y_0\sim q_T\). Then \((Y_t)_{0\leq t\leq T}\) under \(Q\) has the same law as \((X_{T-t})_{0\leq t\leq T}\). In particular, for every integrable measurable function \(\phi\),
\[
\E_Q[\phi(t,Y_t)]=\E_{X_{T-t}}[\phi(t,X_{T-t})].
\]
Here and below, \(\E_Q\) denotes expectation over the reverse path \(Y\) under \(Q\), whereas subscripts involving \(X\) identify the forward random variables being integrated.

For a reverse-time grid \(0=t_0<\cdots<t_N=T-t_{\min}\), define
\[
u_i:=T-t_i,
\qquad
\eta_i:=u_i-u_{i+1}=t_{i+1}-t_i.
\]
Thus \(T=u_0>\cdots>u_N=t_{\min}\).

\begin{assu}[Relative mesh]\label{assu:relative-grid}
Suppose \(0<t_{\min}\leq1\leq T\) and, for some \(0<\kappa\leq1\),
\begin{equation}\label{eq:relative-grid}
\eta_i\leq\kappa\min\{1,u_{i+1}\},
\qquad i=0,\ldots,N-1.
\end{equation}
\end{assu}

For \(N\)-step complexity bounds, we use the rate-regular subclass satisfying
\begin{equation}\label{eq:relative-grid-rate}
\kappa\leq\frac{C_{\mathrm{grid}}}{N}
\left(T+\log\frac1{t_{\min}}\right),
\end{equation}
where \(C_{\mathrm{grid}}\) is independent of
\(d,k,N,T,t_{\min}\). As a concrete example, consider the hybrid grid that is
uniform in
\[
\psi(u):=
\begin{cases}
\log u,&0<u\leq1,\\
u-1,&u\geq1,
\end{cases}
\qquad
\psi(u_i)=\psi(T)
-\frac{i}{N}\left(T-1+\log\frac1{t_{\min}}\right).
\]
Write
\(\Delta_N:=\{T-1+\log(1/t_{\min})\}/N\). Its intervals satisfy
\(\eta_i=\Delta_N\) above \(u=1\) and
\(\eta_i=(e^{\Delta_N}-1)u_{i+1}\) below \(u=1\); the crossing interval obeys
the same relative bound up to a universal constant. Consequently, if
\(\Delta_N\leq c_0\) for a sufficiently small universal constant \(c_0\),
then Assumption~\ref{assu:relative-grid} holds with
\begin{equation}\label{eq:hybrid-relative-rate}
\kappa=C\Delta_N\leq1,
\qquad
\kappa
\leq\frac{C}{N}\left(T+\log\frac1{t_{\min}}\right).
\end{equation}

All path-space KL comparisons below assume that the displayed drift-difference
energies are finite. If a direct Novikov condition is unavailable, Girsanov's
formula follows by standard stopping-time truncation.

For a learned score \(s_\theta\) and a learned posterior mean \(\widehat m\),
consider the two discretized reverse processes, for
\(t\in[t_i,t_{i+1})\),
\begin{align}
d\widehat Y_t^{\mathrm{score}}
&=
\left\{
\frac12\bbeta(u)\widehat Y_t^{\mathrm{score}}
+\bbeta(u)s_\theta(\widehat Y_{t_i}^{\mathrm{score}},u_i)
\right\}dt
+\bbeta(u)^{1/2}d\widetilde B_t,
\qquad
\widehat Y_0^{\mathrm{score}}\sim\pi,
\label{eq:score-frozen-sampler}\\
d\widehat Y_t^{\mathrm{pm}}
&=
\left\{
L_u\widehat Y_t^{\mathrm{pm}}
+D_u\widehat m_{u_i}(\widehat Y_{t_i}^{\mathrm{pm}})
\right\}dt
+\bbeta(u)^{1/2}d\widetilde B_t,
\qquad
\widehat Y_0^{\mathrm{pm}}\sim\pi,
\label{eq:ld-sampler}
\end{align}
where \(u=T-t\). Denote their output laws by \(p_{T-t_{\min}}\) and
\(\widehat p_{T-t_{\min}}\), respectively. The score sampler freezes the whole
learned score. The posterior-mean sampler keeps \(L_u,D_u\) time dependent and
freezes only the learned clean-data predictor. Although the two prediction
targets are related pointwise by \eqref{eq:ld-tweedie}, these freezing rules
define different numerical schemes; no uniform ordering between their errors
is assumed.

Their integrated approximation errors are
\begin{align}
\operatorname{Err}_{\mathrm{score}}
&:=
\sum_{i=0}^{N-1}\int_{u_{i+1}}^{u_i}
\E_{X_{u_i}}
\bigl\|h_{u_i}(X_{u_i})-s_\theta(X_{u_i},u_i)\bigr\|_{\bbeta(u)}^2du
\leq\epsilon_{\mathrm{score}}^2,
\label{eq:score-error-budget}\\
\operatorname{Err}_{\mathrm{pm}}
&:=
\frac12\sum_{i=0}^{N-1}\int_{u_{i+1}}^{u_i}
\E_{X_{u_i}}
\bigl\|m_{u_i}(X_{u_i})-\widehat m_{u_i}(X_{u_i})\bigr\|_{W_u}^2du
\leq\epsilon_{\mathrm{pm}}^2.
\label{eq:ld-estimation-error}
\end{align}
These are input error budgets, not regularity assumptions. In particular,
\(\operatorname{Err}_{\mathrm{pm}}\) measures posterior-mean approximation,
not score approximation.

Under the joint forward law, define the corresponding exact-prediction
discretization errors by
\begin{align}
\operatorname{Disc}_{\mathrm{score}}(\mathcal T_N)
&:=
\sum_{i=0}^{N-1}\int_{u_{i+1}}^{u_i}
\E
\left[
\left\|h_u(X_u)-h_{u_i}(X_{u_i})\right\|_{\bbeta(u)}^2
\right]du,
\label{eq:score-disc-definition}\\
\operatorname{Disc}_{\mathrm{pm}}(\mathcal T_N)
&:=
\frac12\sum_{i=0}^{N-1}\int_{u_{i+1}}^{u_i}
\E
\left[
\left\|m_u(X_u)-m_{u_i}(X_{u_i})\right\|_{W_u}^2
\right]du.
\label{eq:ld-disc}
\end{align}

For the posterior-mean branch, the Markov property and iterated conditioning
along \(X_0\to X_u\to X_v\) give, for \(0<u\leq v\),
\begin{equation}\label{eq:ld-markov-identity}
\E[m_u(X_u)\mid X_v]=m_v(X_v).
\end{equation}

The score decomposition is standard
\citep[proof of Theorem~1]{benton2023nearly}. The posterior-mean decomposition
uses the same Girsanov argument together with the forward Markov identity in
\eqref{eq:ld-markov-identity}. Both proofs are given in
Appendix~\ref{app:kl-decomposition-proof}.
\begin{theo}[KL error decompositions]\label{theo:kl-decomposition}
If the approximation budgets above hold, then
\begin{align}
\KL(q_{t_{\min}}\|p_{T-t_{\min}})
&\leq
\epsilon_{\mathrm{score}}^2
+\operatorname{Disc}_{\mathrm{score}}(\mathcal T_N)
+\KL(q_T\|\pi),
\label{eq:score-kl-decomposition}\\
\KL(q_{t_{\min}}\|\widehat p_{T-t_{\min}})
&\leq
\epsilon_{\mathrm{pm}}^2
+\operatorname{Disc}_{\mathrm{pm}}(\mathcal T_N)
+\KL(q_T\|\pi).
\label{eq:ld-kl-decomposition}
\end{align}
\end{theo}

\paragraph{Low-dimensional support.}
The intrinsic analysis in Section~4 uses the following metric-entropy model;
these assumptions are not imposed in Section~3. Following
\cite{huang2026ddpm}, let \(\cX:=\operatorname{supp}(q_0)\subset\R^d\) be the
closure of the intersection of all measurable sets \(\cX'\) with
\(\Pr(X_0\in\cX')=1\). For \(x\in\R^d\) and \(r>0\), write
\(B(x,r):=\{y\in\R^d:\|y-x\|_2\leq r\}\). The Euclidean covering number
\(N_{\mathrm{cover}}(\cX,\|\cdot\|_2,\varepsilon)\) is the smallest number of
radius-\(\varepsilon\) balls whose union contains \(\cX\), and
\begin{equation}\label{eq:metric-entropy-definition}
\mathsf H_{\cX}(\varepsilon)
:=\log N_{\mathrm{cover}}(\cX,\|\cdot\|_2,\varepsilon)
\end{equation}
is its metric entropy \citep[Chapter~5]{wainwright2019high}.

\begin{assu}[Intrinsic dimension]\label{assu:ld-entropy}
Let \(k\geq2\), let \(C_0>0\) be a sufficiently large universal constant, and
set \(\varepsilon_0:=k^{-C_0}\). For a universal constant
\(C_{\mathrm{cover}}>0\), assume
\begin{equation}\label{eq:ld-covering}
\mathsf H_{\cX}(\varepsilon_0)
\leq C_{\mathrm{cover}}k\log\frac1{\varepsilon_0}.
\end{equation}
\end{assu}

\begin{assu}[Bounded support]\label{assu:ld-radius}
For a universal constant \(C_R>0\), assume
\begin{equation}\label{eq:ld-radius}
\sup_{x\in\cX}\|x\|_2\leq R,
\qquad R:=k^{C_R}.
\end{equation}
\end{assu}

These assumptions match the low-dimensional setting of
\cite{huang2026ddpm}. Assumption~\ref{assu:ld-entropy} controls the number of
distinguishable data points at the fixed resolution \(\varepsilon_0\), while
Assumption~\ref{assu:ld-radius} is used only to express the final step
complexity in terms of \(k\). For context, the doubling dimension of a set
\(\mathcal M\) is the smallest \(s\) such that every
\(B(x,r)\cap\mathcal M\) can be covered by \(2^s\) balls of radius \(r/2\).
If \(\cX\) lies within distance \(\varepsilon_0\) of a bounded set of doubling
dimension \(k\), then \eqref{eq:ld-covering} follows
\citep[Lemma~2]{huang2026ddpm}. Bounded subsets of \(k\)-dimensional linear
subspaces satisfy the same condition by a volumetric bound
\citep[Section~4.2.1]{vershynin2018high}.

\subsection{Information identities}

Let \(\gamma\) be the density of \(\pi\). For \(u>0\), define
\[
r_u(x):=\frac{q_u(x)}{\gamma(x)},\qquad
\zeta_u(x):=\nabla\log r_u(x),\qquad
H_u(x):=\nabla^2\log r_u(x).
\]

The cumulative estimates use the dissipation of relative entropy along the forward OU flow. For a symmetric matrix \(A\), define the weighted relative Fisher functional
\begin{equation}\label{eq:IA}
I_A(u):=\int q_u(x)\left\langle A\nabla\log r_u(x),\nabla\log r_u(x)\right\rangle dx.
\end{equation}
The choices used below are \(A=\bbeta(u)\), \(A=\bbeta'(u)\), and \(A=\bbeta(u)^2\). We also define
\begin{equation}\label{eq:Jbeta}
J_\bbeta(u):=\int q_u(x)\left\|\bbeta(u)^{1/2}H_u(x)\bbeta(u)^{1/2}\right\|_F^2dx.
\end{equation}
All these quantities are evaluated at positive forward times, after the OU transition has smoothed the data law. Thus no Fisher-information assumption is imposed on \(q_0\).

The infinitesimal generator of the forward equation relative to the Gaussian reference measure is
\begin{equation}\label{eq:Lt}
\mathcal L_u\phi=\frac12\Tr(\bbeta(u)\nabla^2\phi)-\frac12\langle \bbeta(u)x,\nabla\phi\rangle.
\end{equation}

The forward equation and Gaussian integration by parts give the two dissipation identities below. Their derivations are collected in Appendix~\ref{app:technical-identities}.

\begin{prop}[KL dissipation]\label{prop:kl-dissipation}
For \(u>0\),
$$\frac{d}{du}\KL(q_u\|\pi)=-\frac12I_\bbeta(u).$$
\end{prop}

\begin{prop}[Fisher-information evolution]\label{prop:fisher-evolution}
On each smooth interval of \(\bbeta\),
$$\frac{d}{du}I_\bbeta(u)=I_{\bbeta'}(u)-I_{\bbeta^2}(u)-J_\bbeta(u).$$
\end{prop}

\section{Ambient-dimensional analysis under score freezing}
\label{sec:ambient-analysis}

Theorem~\ref{theo:kl-decomposition} reduces the score-frozen analysis to
\(\operatorname{Disc}_{\mathrm{score}}(\mathcal T_N)\). Following the framework
outlined in the Introduction, we first identify the local freezing error, then
accumulate it through the forward evolution, and finally combine the result
with the common KL decomposition. The terminal initialization term is treated
only at the last step because it is independent of the local analysis.

\begin{mainassu}[Ambient cumulative regime]\label{assu:cumulative}
Suppose
\[
M_2\leq C_{\mathrm{mom}}d,
\qquad
\bbeta'(u)\preceq0,
\]
where the derivative condition holds on each smooth interval.
\end{mainassu}
The monotonicity condition simplifies the cumulative proof;
Appendix~\ref{app:cumulative-auxiliary} gives the corresponding
integrating-factor argument for nonmonotone schedules.
Throughout Section~3, the standing assumptions are
Assumptions~\ref{assu:baseline},~\ref{assu:relative-grid}, and~\ref{assu:cumulative}.

\paragraph{Constant convention.}
In Section~3 and its appendix proofs, unsubscripted \(C,c>0\) may depend only
on \(C_{\mathrm{mom}},m_\bbeta,B_\bbeta,L_\bbeta\). Unless stated otherwise, they are
independent of \(d,N,T,t_{\min}\), the grid, the learned score, and the data
law beyond \(M_2\leq C_{\mathrm{mom}}d\). The constants implicit in
\(\lesssim,\gtrsim,\asymp\) follow the same convention.

\subsection{Local discretization error analysis}

By time reversal, the local error from freezing \(h_{T-s}(Y_s)\) at reverse
time \(s\) can be evaluated under the joint law of two forward marginals:
\begin{equation}\label{eq:local-E-definition}
E_{s,t}
:=
\E_{X_{T-t},X_{T-s}}
\bigl\|h_{T-t}(X_{T-t})-h_{T-s}(X_{T-s})\bigr\|_{\bbeta(T-t)}^2,
\qquad 0\leq s<t<T.
\end{equation}
The local estimate replaces this two-time quantity by the one-time weighted score energy
\begin{equation}\label{eq:K-definition}
K(t):=\E_{X_{T-t}}\|h_{T-t}(X_{T-t})\|_{\bbeta(T-t)}^2.
\end{equation}

The schedule dependence of the local comparison is encoded by
\begin{align}
C_{s,t}
&:=
\bbeta(T-s)+\bbeta(T-t)
-\Phi(T-s,T-t)\bbeta(T-t)
-\bbeta(T-t)\Phi(T-s,T-t)^T,\notag\\
M_{s,t}
&:=
\bbeta(T-s)^{-1/2}C_{s,t}
\bbeta(T-s)^{-1/2}.
\label{eq:M-definition}
\end{align}
Both matrices are explicit once the schedule and its transition matrix are fixed.

The following conditional-score identity is the key input.
\begin{lemm}[Conditional score identity]\label{lem:conditional-score}
For the transition matrix \(\Phi\) defined in \eqref{eq:Phi} and \(0\leq s<t<T\),
    $$\E_{X_{T-t}\mid X_{T-s}}[h_{T-t}(X_{T-t})]=\Phi(T-s,T-t)^Th_{T-s}(X_{T-s}).$$
\end{lemm}
\begin{proof}
Let $p(x,y)$ be the transition density from $T-t$ to $T-s$ and set
$\Phi_{s,t}:=\Phi(T-s,T-t)$. Its Gaussian form gives
\(\nabla_xp(x,y)=-\Phi_{s,t}^T\nabla_yp(x,y)\). Since
\(h_{T-t}(x)q_{T-t}(x)=\nabla_xq_{T-t}(x)\), integration by parts in \(x\)
gives, with all integrals over \(\R^d\),
\[
\begin{aligned}
\int h_{T-t}p\,q_{T-t}\,dx
&=\int p\,\nabla_xq_{T-t}\,dx
=-\int q_{T-t}\nabla_xp\,dx\\
&=\Phi_{s,t}^T\int q_{T-t}\nabla_yp\,dx
=\Phi_{s,t}^T\nabla_y\int p\,q_{T-t}\,dx
=\Phi_{s,t}^T\nabla_yq_{T-s}(y).
\end{aligned}
\]
The last equality is the forward marginalization identity. Moreover,
\[
\mathbb P(X_{T-t}\in dx\mid X_{T-s}=y)
=\frac{p(x,y)q_{T-t}(x)}{q_{T-s}(y)}\,dx.
\]
Dividing the displayed equality by \(q_{T-s}(y)\) therefore proves the
conditional-score identity. The integration by parts and differentiation under
the integral can be justified with compactly supported cutoffs; Gaussian
smoothing at \(T-t>0\) and the nondegenerate Gaussian transition provide the
required integrability when the cutoffs are removed.
\end{proof}

\begin{lemm}[Local sandwich bound]\label{lem:local-sandwich}
For \(0\leq s<t<T\),
\[
K(t)-\{1-\lambda_{\min}(M_{s,t})\}K(s)
\leq E_{s,t}\leq
K(t)-\{1-\lambda_{\max}(M_{s,t})\}K(s).
\]
\end{lemm}

\begin{proof}
Write $h_s=h_{T-s}(X_{T-s})$ and $h_t=h_{T-t}(X_{T-t})$. For brevity, set
\[
B_s:=\bbeta(T-s),\qquad
B_t:=\bbeta(T-t),\qquad
\Phi_{s,t}:=\Phi(T-s,T-t).
\]
Expanding the square and applying Lemma~\ref{lem:conditional-score} to the cross term gives
\begin{equation}\label{eq:local-error-identity}
\begin{split}
E_{s,t}
&=K(t)+\E_{X_{T-s}}\left[h_s^TB_th_s\right]
-2\E_{X_{T-s}}\left[h_s^T\Phi_{s,t}B_th_s\right]\\
&=K(t)-K(s)+\E_{X_{T-s}}\left[h_s^TC_{s,t}h_s\right].
\end{split}
\end{equation}
Since $B_t$ is symmetric, for every $x\in\R^d$,
\[
2x^T\Phi_{s,t}B_tx
=x^T\left(\Phi_{s,t}B_t+B_t\Phi_{s,t}^T\right)x.
\]
The definition of $C_{s,t}$ therefore gives
\[
x^TB_tx-2x^T\Phi_{s,t}B_tx
=-x^TB_sx+x^TC_{s,t}x,
\]
which proves \eqref{eq:local-error-identity}. Letting $y=B_s^{1/2}x$, we have
\[
x^TC_{s,t}x=y^TM_{s,t}y,
\qquad
x^TB_sx=\|y\|^2.
\]
The Rayleigh quotient gives
$$
\lambda_{\min}(M_{s,t})\,x^T\bbeta(T-s)x
\leq x^TC_{s,t}x
\leq \lambda_{\max}(M_{s,t})\,x^T\bbeta(T-s)x.
$$
Taking $x=h_s$ and expectations in \eqref{eq:local-error-identity} proves both bounds.
\end{proof}

The local estimate separates the change in weighted score energy from a matrix
correction determined by the forward transition and the schedule geometry. We
next accumulate this local error along the forward diffusion.

\subsection{Global discretization error analysis via the forward SDE}

We first connect $K$ directly to the forward information quantities. Since
\(\nabla\log\gamma(x)=-x\), we have
$$\nabla\log r_u(x)=\nabla\log q_u(x)-\nabla\log\gamma(x)=h_u(x)+x.$$

Therefore
\begin{equation}
\begin{split}
I_\bbeta(T-t)
&=\E_{X_{T-t}}\left[(h_{T-t}(X_{T-t})+X_{T-t})^T\bbeta(T-t)(h_{T-t}(X_{T-t})+X_{T-t})\right] \\
&=\E_{X_{T-t}}\left[h_{T-t}(X_{T-t})^T\bbeta(T-t)h_{T-t}(X_{T-t})\right]+2\E_{X_{T-t}}\left[X_{T-t}^T\bbeta(T-t)h_{T-t}(X_{T-t})\right]\\
&\quad +\E_{X_{T-t}}\left[X_{T-t}^T\bbeta(T-t)X_{T-t}\right].
\end{split}
\end{equation}

For every forward time $u>0$, Gaussian smoothing makes $q_u$ smooth. A coordinatewise integration by parts, justified by a cutoff argument and $M_2<\infty$, gives
\begin{align*}
\E_{X_u}\left[X_u^T\bbeta(u)h_u(X_u)\right]
&=\sum_{i,j=1}^d\bbeta_{ij}(u)
\int_{\R^d}x_i\,\partial_jq_u(x)\,dx\\
&=-\sum_{i,j=1}^d\bbeta_{ij}(u)\delta_{ij}
=-\Tr(\bbeta(u)).
\end{align*}

Since
\(\E[X_u^T\bbeta(u)X_u]=\Tr\{\bbeta(u)\Sigma_u\}\), evaluating the
preceding identity at \(u=T-t\) gives, for \(0\leq t<T\),
\begin{equation}\label{eq:K-information}
K(t)=I_\bbeta(T-t)+2\Tr(\bbeta(T-t))-\Tr(\bbeta(T-t)\Sigma_{T-t}).
\end{equation}
This identity links the local reverse-time error to the forward KL dissipation in Proposition~\ref{prop:kl-dissipation}.
For later use, define the forward-time remainder
\begin{equation}\label{eq:G-definition}
G(u):=2\Tr\bbeta(u)-\Tr(\bbeta(u)\Sigma_u).
\end{equation}

The cumulative proof uses the following bounds, all expressed under the
unconditional forward law.
\begin{lemm}[Forward information and schedule bounds]
\label{lem:cumulative-technical-bounds}
Under Assumption~\ref{assu:baseline}, with $M_2\leq C_{\mathrm{mom}}d$, for every
$u\in[t_{\min},T]$,
\[
I_\bbeta(u)\leq Cd\left(1+\frac1u\right),
\qquad
K(T-u)\leq Cd\left(1+\frac1u\right),
\]
and
\[
|G(u)-G(v)|\leq Cd|u-v|,
\qquad u,v\in[t_{\min},T],
\]
and, for $0\leq s<t\leq T-t_{\min}$,
\[
\|M_{s,t}\|_2
\leq\frac{L_\bbeta+B_\bbeta^2}{m_\bbeta}(t-s),
\qquad
\lambda_{\max}(M_{s,t})\leq\|M_{s,t}\|_2.
\]
Here \(C\) depends only on \(C_{\mathrm{mom}},m_\bbeta,B_\bbeta,L_\bbeta\). In
particular, it is independent of \(d,N,T,t_{\min}\) and of \(q_0\) beyond
\(M_2\leq C_{\mathrm{mom}}d\).
\end{lemm}

These estimates follow directly from the forward equation and require no
localization argument; see Appendix~\ref{app:cumulative-auxiliary}. Moreover,
schedule monotonicity and
Proposition~\ref{prop:fisher-evolution} give, on each smooth interval,
\[
I_\bbeta'(u)
=I_{\bbeta'}(u)-I_{\bbeta^2}(u)-J_\bbeta(u)\leq0.
\]
Hence \(I_\bbeta\) is nonincreasing on \([t_{\min},T]\). Together with the
relative mesh condition in Assumption~\ref{assu:relative-grid}, these forward
bounds allow the local estimate to be accumulated over the grid. For a
nonmonotone schedule, an integrating-factor modification of \(I_\bbeta\)
remains nonincreasing, and the same argument incurs only the additional factor
\(e^{A_\bbeta}\); see Remark~\ref{rema:nonmonotone-schedules}.

\begin{theo}[Cumulative discretization bound]\label{theo:cumulative-E-relative-grid}
Under the standing assumptions, recall that \(\kappa\) is the relative-mesh
constant in Assumption~\ref{assu:relative-grid}. Then
$$
\sum_{i=0}^{N-1}\int_{t_i}^{t_{i+1}}E_{t_i,t}\,dt
\leq
Cd\kappa
\left(
1+T+\log\frac1{t_{\min}}
\right).
$$
\end{theo}

\begin{proof}
By Lemma~\ref{lem:local-sandwich},
$$
E_{t_i,t}\leq K(t)-K(t_i)+\lambda_{\max}(M_{t_i,t})K(t_i).
$$
Using the forward-time variable \(u=T-t\), equation~\eqref{eq:K-information} gives
$$
K(T-u)=I_\bbeta(u)+G(u).
$$
Since \(I_\bbeta\) is nonincreasing, it induces the positive Stieltjes measure
\(\mu_I(du):=-dI_\bbeta(u)\). The Fisher-information part satisfies
$$
\begin{aligned}
\sum_i\int_{u_{i+1}}^{u_i}\left(I_\bbeta(u)-I_\bbeta(u_i)\right)du
&=
\sum_i\int_{(u_{i+1},u_i]}(v-u_{i+1})\,\mu_I(dv)\\
&\leq
\kappa\int_{(t_{\min},T]}\min\{1,v\}\,\mu_I(dv).
\end{aligned}
$$
The equality follows from the Stieltjes representation
$I_\bbeta(u)-I_\bbeta(u_i)=\mu_I((u,u_i])$
and Fubini's theorem. The final inequality uses
$v-u_{i+1}\leq \eta_i\leq\kappa\min\{1,u_{i+1}\}\leq\kappa\min\{1,v\}$.
To make the endpoint contribution explicit, set $w(v):=\min\{1,v\}$ and
$J:=\int_{(t_{\min},T]}w(v)\,\mu_I(dv)$. Since $t_{\min}\leq1\leq T$,
we have $w(t_{\min})=t_{\min}$, $w(T)=1$, and
$w'(v)=\mathbf 1_{(t_{\min},1)}(v)$ almost everywhere. Integration by parts gives
\begin{align*}
J
&=t_{\min}I_\bbeta(t_{\min})-I_\bbeta(T)
+\int_{t_{\min}}^1 I_\bbeta(v)\,dv\\
&\leq Cd+Cd\int_{t_{\min}}^1\left(1+\frac1v\right)dv
\leq Cd\left(1+\log\frac1{t_{\min}}\right).
\end{align*}
Indeed, since $t_{\min}\leq1$, Lemma~\ref{lem:cumulative-technical-bounds} gives
$t_{\min}I_\bbeta(t_{\min})\leq
Cd(1+t_{\min})\leq2Cd$, while
$I_\bbeta(T)\geq0$. By the constant convention above, this \(C\) is independent
of \(t_{\min}\) and of the data law beyond \(M_2\leq C_{\mathrm{mom}}d\).
Hence the Fisher-information contribution is bounded by $Cd\kappa(1+\log(1/t_{\min}))$.

For the remaining part of \(K\), the global Lipschitz estimate in Lemma~\ref{lem:cumulative-technical-bounds} gives
\[
\left|\int_{u_{i+1}}^{u_i}\left(G(u)-G(u_i)\right)du\right|
\leq Cd\int_{u_{i+1}}^{u_i}(u_i-u)\,du
=\frac{Cd}{2}\eta_i^2.
\]
Since \(\eta_i\leq\kappa\), we have \(\sum_i \eta_i^2\leq \kappa\sum_i \eta_i\leq\kappa T\). Thus the total \(G\)-remainder is bounded by \(Cd\kappa T\).

It remains to control the \(M\)-term. By Lemma~\ref{lem:cumulative-technical-bounds},
\[
\lambda_{\max}(M_{t_i,t})
\leq
\|M_{t_i,t}\|_2
\leq
\frac{L_\bbeta+B_\bbeta^2}{m_\bbeta}(t-t_i),
\qquad
K(t_i)\leq Cd\left(1+\frac1{u_i}\right).
\]
Since
\[
\int_{t_i}^{t_{i+1}}(t-t_i)\,dt
=\frac12\eta_i^2,
\]
we obtain
\[
\sum_iK(t_i)\int_{t_i}^{t_{i+1}}\lambda_{\max}(M_{t_i,t})dt
\leq
\frac{L_\bbeta+B_\bbeta^2}{2m_\bbeta}
\sum_iK(t_i)\eta_i^2
\leq
Cd\sum_i\left(1+\frac1{u_i}\right)\eta_i^2.
\]
The first sum is at most \(Cd\kappa T\). For the second, if \(u_{i+1}\geq1\), then \(u_i\geq1\) and \(\eta_i^2/u_i\leq \eta_i^2\leq\kappa \eta_i\). If \(u_{i+1}<1\), then \(\eta_i\leq\kappa u_{i+1}\leq\kappa u_i\), so again \(\eta_i^2/u_i\leq\kappa \eta_i\). Therefore
$$
\sum_i\frac{\eta_i^2}{u_i}\leq \kappa\sum_i \eta_i\leq \kappa T.
$$
Combining the three contributions proves the cumulative estimate.
\end{proof}

The mesh condition permits intervals of order \(\kappa\) away from the
early-stopping endpoint and requires intervals of order \(\kappa u\) near it.
The factor \(\log(1/t_{\min})\) records the cost of approaching the unsmoothed
data law.

\subsection{Overall KL guarantee and step complexity}

Theorem~\ref{theo:cumulative-E-relative-grid} controls discretization. The
remaining initialization error is caused by replacing the terminal law
\(q_T\) with \(\pi\), and is governed by the forward mixing rate.

\begin{lemm}[Terminal mixing bound]\label{lem:terminal-mixing}
Under Assumption~\ref{assu:baseline}, let
\(\Lambda_T:=\int_0^T\lambda_{\min}(\bbeta(u))\,du\). Then
\begin{equation}\label{eq:terminal-mixing-rate}
\tau_T
:=\frac{M_2}{2}e^{-\Lambda_T}
+\frac{d}{4}\frac{e^{-2\Lambda_T}}{1-e^{-\Lambda_T}},
\qquad
\KL(q_T\|\pi)\leq\tau_T.
\end{equation}
\end{lemm}
Lemma~\ref{lem:terminal-mixing} quantifies the forward mixing toward \(\pi\);
its proof is given in Appendix~\ref{app:terminal-initialization}. For
\(\Lambda_T\geq\log 2\), it gives
\(\tau_T\leq M_2e^{-\Lambda_T}/2+de^{-2\Lambda_T}/2\). Since
\(\Lambda_T\geq m_\bbeta T\), a terminal horizon of order
\(m_\bbeta^{-1}\log((M_2+d)/\varepsilon^2)\) makes
\(\tau_T=O(\varepsilon^2)\). Thus the terminal horizon required to control
initialization grows only logarithmically in \(d/\varepsilon\).
Combining the KL decomposition in Theorem~\ref{theo:kl-decomposition}, the
global discretization estimate in Theorem~\ref{theo:cumulative-E-relative-grid},
and this forward-mixing bound gives the following final KL guarantee.

\begin{theo}[Relative-grid KL guarantee and step complexity]
\label{theo:relative-grid-kl-complexity}
Under the standing assumptions and the score-approximation budget
\eqref{eq:score-error-budget},
\begin{equation}\label{eq:relative-grid-score-error}
\KL(q_{t_{\min}}\|p_{T-t_{\min}})
\leq
\epsilon_{\mathrm{score}}^2+\tau_T
+Cd\kappa\left(1+T+\log\frac1{t_{\min}}\right).
\end{equation}
If the grid is also rate-regular in the sense of
\eqref{eq:relative-grid-rate}, then
\begin{equation}\label{eq:rate-regular-final-error}
\KL(q_{t_{\min}}\|p_{T-t_{\min}})
\leq
\epsilon_{\mathrm{score}}^2+\tau_T
+\frac{Cd}{N}
\left(T+\log\frac1{t_{\min}}\right)^2.
\end{equation}
\begin{samepage}
More precisely, let \(0<\varepsilon\leq1\), suppose
\(\epsilon_{\mathrm{score}}^2\leq\varepsilon^2/3\), and take
\(T=\max\{1,m_\bbeta^{-1}\log(4(M_2+d)/\varepsilon^2)\}\). For a
sufficiently large constant \(C\), the choice
\begin{equation}\label{eq:required-N-rate-regular}
\begin{split}
N
&=\left\lceil
\frac{Cd}{\varepsilon^2}
\left(T+\log\frac1{t_{\min}}\right)^2
\right\rceil \\
&=O\!\left(
\frac{d}{\varepsilon^2}
\left[1+\log\frac{d}{\varepsilon^2}
+\log\frac1{t_{\min}}\right]^2
\right)
\end{split}
\end{equation}
guarantees
\(\KL(q_{t_{\min}}\|p_{T-t_{\min}})\leq\varepsilon^2\). The constant in
the last line depends only on the standing schedule and moment bounds.
\end{samepage}
\end{theo}
The proof is given in Appendix~\ref{app:cumulative-auxiliary}. The KL target is
the early-stopped law \(q_{t_{\min}}\). Under only a finite-second-moment
assumption, \(q_0\) may be singular, so no finite KL comparison between
\(q_0\) and \(q_{t_{\min}}\) is available in general. The early-stopping bias
can instead be measured in Wasserstein distance. The forward Gaussian coupling
gives \(W_2(q_0,q_{t_{\min}})\leq C\sqrt{dt_{\min}}\), so
\(t_{\min}\lesssim\delta^2/d\) suffices for a \(W_2\) bias of at most
\(\delta\). This statement is separate from the KL guarantee above, which
keeps \(t_{\min}\) explicit, as in the scalar result
\citep[Corollary~1]{benton2023nearly}.

\section{Intrinsic-dimensional analysis under posterior-mean freezing}
\label{sec:intrinsic-extension}

Section~2 defines the posterior-mean-frozen sampler, the low-dimensional
support assumptions, and the corresponding KL decomposition. We now follow the
same order as in Section~3: identify the local freezing error, accumulate it on
the forward time axis, and then return to the KL guarantee. Only the quantities
used in the first two steps change. Here the local error is a posterior-mean
increment, which permits intrinsic-dimensional control of
\(\operatorname{Disc}_{\mathrm{pm}}(\mathcal T_N)\).

The general discretization bound below is stated at an arbitrary resolution
and keeps the resolution error explicit. Its \(k\log k\) specialization uses
\(\varepsilon_0\) and states the required relation between this resolution and
\(t_{\min}\). The only learning input is the posterior-mean approximation budget
\(\operatorname{Err}_{\mathrm{pm}}\) in \eqref{eq:ld-estimation-error}.
Throughout this section, unsubscripted \(C,c>0\) may depend only on
\(m_\bbeta,B_\bbeta,L_\bbeta,C_0,C_{\mathrm{cover}}\), and not on
\(d,k,N,T,t_{\min}\), the grid, or other features of the data law. The final
complexity constant may also depend on \(C_R\) and the polynomial
ambient-dimension exponent.

\subsection{Local discretization error analysis}

In parallel with \eqref{eq:local-E-definition}, define the local error from
freezing the posterior mean at reverse time \(s\) by
\begin{equation}\label{eq:ld-local-E-definition}
E_{s,t}^{\mathrm{pm}}
:=
\E_{X_{T-t},X_{T-s}}
\|m_{T-t}(X_{T-t})-m_{T-s}(X_{T-s})\|_{W_{T-t}}^2,
\qquad 0\leq s<t<T.
\end{equation}
Then the discretization term in \eqref{eq:ld-disc} is
\begin{equation}\label{eq:ld-disc-reverse-time}
\operatorname{Disc}_{\mathrm{pm}}(\mathcal T_N)
=\frac12\sum_{i=0}^{N-1}
\int_{t_i}^{t_{i+1}}E_{t_i,t}^{\mathrm{pm}}\,dt.
\end{equation}
Define the posterior covariance and its forward average by
\[
P_u(x):=\Cov(X_0\mid X_u=x),
\qquad
\Pbar_u:=\E[P_u(X_u)],
\]
and define the marginal second moment of the posterior mean by
\[
\Rbar_u:=\E[m_u(X_u)m_u(X_u)^T].
\]
The conditional second-moment decomposition gives
\begin{equation}\label{eq:ld-second-moment-split}
\Sigma_0=\Pbar_u+\Rbar_u.
\end{equation}

\begin{samepage}
\begin{theo}[Local posterior-mean discretization error]
\label{theo:ld-local-error}
For \(0\leq s<t<T\), the local error in
\eqref{eq:ld-local-E-definition} has the exact forward representation
\begin{equation}\label{eq:ld-local-forward}
\begin{aligned}
0\leq E_{s,t}^{\mathrm{pm}}
&=\Tr\left[W_{T-t}\{\Pbar_{T-s}-\Pbar_{T-t}\}\right]\\
&=\Tr\left[W_{T-t}\{\Rbar_{T-t}-\Rbar_{T-s}\}\right].
\end{aligned}
\end{equation}
In particular, for \(t\in[t_i,t_{i+1}]\),
\begin{equation}\label{eq:ld-local-endpoint}
E_{t_i,t}^{\mathrm{pm}}
\leq
\Tr\left[
W_{T-t}\{\Pbar_{u_i}-\Pbar_{u_{i+1}}\}
\right].
\end{equation}
\end{theo}
\end{samepage}

\begin{proof}
For \(0<u<v\), the forward Markov chain
\(X_0\to X_u\to X_v\) gives
\(\E[m_u(X_u)\mid X_v]=m_v(X_v)\). Expanding the square and applying
this identity to the cross terms yields
\begin{equation}\label{eq:ld-increment}
\begin{aligned}
&\E\left[
\{m_u(X_u)-m_v(X_v)\}
\{m_u(X_u)-m_v(X_v)\}^T
\right]\\
&\qquad=\Rbar_u-\Rbar_v
=\Pbar_v-\Pbar_u\succeq0.
\end{aligned}
\end{equation}
The second equality follows from \eqref{eq:ld-second-moment-split}. Set
\(u=T-t\) and \(v=T-s\), and take the trace against \(W_u\succeq0\)
to obtain \eqref{eq:ld-local-forward}. If \(t\in[t_i,t_{i+1}]\), then
\(u_{i+1}\leq u\leq u_i\), so monotonicity in
\eqref{eq:ld-increment} gives
\(\Pbar_{u_i}-\Pbar_u\preceq
\Pbar_{u_i}-\Pbar_{u_{i+1}}\), proving
\eqref{eq:ld-local-endpoint}.
\end{proof}

The two forms in \eqref{eq:ld-local-forward} serve different purposes. The
posterior-covariance form is used for the intrinsic-dimensional bound. The
\(\Rbar\)-form involves only one-time marginal moments of the posterior-mean
predictions. Since \(\E[m_u(X_u)]=\E[X_0]\) for every \(u\), differences of
\(\Rbar_u\) are also differences of the marginal covariance matrices of
\(m_u(X_u)\). They can therefore be estimated from forward-noised samples and
used in the adaptive-grid construction of Section~5. No evolution equation for
posterior covariance is used.

\subsection{Global discretization error via covariance increments}

As in Section~3.2, the cumulative error is an integral against a positive
Stieltjes measure. Here that measure is
\(d\Pbar_u=-d\Rbar_u\succeq0\). Metric entropy controls its weighted mass, while
the schedule comparison below has dimension-free constants. The forward
process supplies the Markov identity \eqref{eq:ld-increment} and a fixed-time
Gaussian corruption model; no evolution equation for \(\Pbar_u\) is used.
Define
\begin{equation}\label{eq:ld-forward-risk}
\mathcal I_{\mathrm{pm}}(u)
:=\Tr(S_u\Pbar_u).
\end{equation}
The matrix \(S_u\) defines the schedule-dependent metric
\begin{equation}\label{eq:ld-anisotropic-metric}
d_{\bbeta,u}(x,y)
:=\|x-y\|_{S_u}
=\|\Gamma_u^{-1/2}A_u(x-y)\|_2.
\end{equation}
Let
\begin{equation}\label{eq:ld-anisotropic-entropy}
\mathsf H_{\bbeta,u}(r)
:=\log N_{\mathrm{cover}}(\cX,d_{\bbeta,u},r),
\qquad
\mathfrak C_{\bbeta}(u)
:=\inf_{r>0}\{\mathsf H_{\bbeta,u}(r)+r^2\}.
\end{equation}
This metric measures which differences between clean samples remain
distinguishable after the forward corruption at time \(u\). The following
fixed-time estimate is the only consequence of low dimensionality needed in
the cumulative argument.

\begin{lemm}[Anisotropic entropy bound]\label{lem:ld-channel}
For every \(u>0\) and \(r>0\) with
\(N_{\mathrm{cover}}(\cX,d_{\bbeta,u},r)<\infty\),
\begin{equation}\label{eq:ld-anisotropic-finite-net}
\mathcal I_{\mathrm{pm}}(u)
\leq C\{\mathsf H_{\bbeta,u}(r)+r^2\}.
\end{equation}
Consequently,
\begin{equation}\label{eq:ld-anisotropic-complexity}
\mathcal I_{\mathrm{pm}}(u)\leq C\mathfrak C_{\bbeta}(u),
\end{equation}
and, for every Euclidean resolution \(\varepsilon>0\),
\begin{equation}\label{eq:ld-finite-net-main}
\mathcal I_{\mathrm{pm}}(u)
\leq C\left\{
\mathsf H_{\cX}(\varepsilon)+\|S_u\|_2\varepsilon^2
\right\}.
\end{equation}
\end{lemm}

The estimate concerns the whitened forward observation
\[
\Gamma_u^{-1/2}X_u=\Gamma_u^{-1/2}A_uX_0+Z.
\]
It compares the posterior mean with a minimum-distance estimator over a finite
net in the metric \eqref{eq:ld-anisotropic-metric}. Since
\(d_{\bbeta,u}(x,y)\leq\|S_u\|_2^{1/2}\|x-y\|_2\), the Euclidean estimate
\eqref{eq:ld-finite-net-main} follows immediately. Covering arguments also
underlie low-dimensional DDPM bounds in \cite{huang2026ddpm}; the anisotropic
fixed-time estimate is proved in
Appendix~\ref{app:intrinsic-proofs}.

\begin{prop}[Continuous forward cumulative reduction]
\label{prop:ld-cumulative-reduction}
Let \(a(u):=\min\{1,u\}\). Under Assumptions~\ref{assu:baseline}
and~\ref{assu:relative-grid},
\begin{equation}\label{eq:ld-S-bounds}
\|S_u\|_2\leq\frac1{e^{m_\bbeta u}-1},
\qquad
W_u\preceq\frac{C}{a(u)}S_u,
\qquad u>0.
\end{equation}
\begin{equation}\label{eq:ld-continuous-reduction}
\operatorname{Disc}_{\mathrm{pm}}
\leq
C\kappa\left\{
\mathcal I_{\mathrm{pm}}(T)
+\int_{t_{\min}}^T
\frac{\mathcal I_{\mathrm{pm}}(u)}{a(u)}\,du
\right\}.
\end{equation}
The constant \(C\) depends only on \(m_\bbeta\) and \(B_\bbeta\), and is
independent of \(d\).
\end{prop}

\begin{proof}
The identity \(\Gamma_u=I-A_uA_u^T\) and the contraction estimate
\(\|A_u\|_2^2\leq e^{-m_\bbeta u}\) give
\[
\Gamma_u\succeq(1-e^{-m_\bbeta u})I,
\qquad
\|S_u\|_2
\leq\frac{e^{-m_\bbeta u}}{1-e^{-m_\bbeta u}}.
\]
Moreover, \eqref{eq:ld-S-W} and
\(\bbeta(u)\preceq B_\bbeta I\) imply
\[
W_u\preceq
\frac{B_\bbeta}{1-e^{-m_\bbeta u}}S_u
\preceq\frac{C}{a(u)}S_u,
\]
which proves \eqref{eq:ld-S-bounds}.

By \eqref{eq:ld-increment}, \(\Pbar_u\) is nondecreasing in the
positive-semidefinite order and therefore induces a positive-semidefinite
matrix-valued Stieltjes measure \(d\Pbar_u\). Fubini's theorem gives
\begin{equation}\label{eq:ld-stieltjes-fubini}
\operatorname{Disc}_{\mathrm{pm}}
=\frac12\sum_{i=0}^{N-1}
\int_{(u_{i+1},u_i]}
\Tr\left[
\{S_{u_{i+1}}-S_r\}\,d\Pbar_r
\right].
\end{equation}
Indeed, the inner integral generated by exchanging the order of integration is
\(\int_{u_{i+1}}^rW_u\,du=S_{u_{i+1}}-S_r\).

The differential bound in \eqref{eq:ld-S-bounds} and
Assumption~\ref{assu:relative-grid} imply, uniformly for
\(r\in[u_{i+1},u_i]\),
\begin{equation}\label{eq:ld-local-S-comparison}
0\preceq S_{u_{i+1}}-S_r\preceq C\kappa S_r.
\end{equation}
To see this, integrate \(W_u\preceq C a(u)^{-1}S_u\) over the interval and
use
\(\int_{u_{i+1}}^{u_i}a(u)^{-1}du\leq C\kappa\); the same differential
inequality gives uniform comparability of \(S_u\) within the interval.
Combining \eqref{eq:ld-stieltjes-fubini} and
\eqref{eq:ld-local-S-comparison} yields
\[
\operatorname{Disc}_{\mathrm{pm}}
\leq C\kappa
\int_{(t_{\min},T]}\Tr(S_u\,d\Pbar_u).
\]
Stieltjes integration by parts and \(S_u'=-W_u\) give
\begin{align*}
\int_{(t_{\min},T]}\Tr(S_u\,d\Pbar_u)
&=\mathcal I_{\mathrm{pm}}(T)
-\mathcal I_{\mathrm{pm}}(t_{\min})
+\int_{t_{\min}}^T\Tr(W_u\Pbar_u)\,du\\
&\leq\mathcal I_{\mathrm{pm}}(T)
+C\int_{t_{\min}}^T
\frac{\mathcal I_{\mathrm{pm}}(u)}{a(u)}\,du,
\end{align*}
where the last line again uses \eqref{eq:ld-S-bounds}. This proves
\eqref{eq:ld-continuous-reduction}.
\end{proof}

\begin{theo}[Cumulative intrinsic-dimensional discretization bound]
\label{theo:ld-cumulative}
Suppose Assumptions~\ref{assu:baseline} and~\ref{assu:relative-grid} hold. Put
\[
L:=T+\log\frac1{t_{\min}}.
\]
Then the schedule-dependent entropy in \eqref{eq:ld-anisotropic-entropy}
satisfies
\begin{equation}\label{eq:ld-main-disc-anisotropic}
\operatorname{Disc}_{\mathrm{pm}}
\leq
C\kappa\left\{
\mathfrak C_{\bbeta}(T)
+\int_{t_{\min}}^T\frac{\mathfrak C_{\bbeta}(u)}{a(u)}\,du
\right\}.
\end{equation}
For every \(\varepsilon>0\) with
\(N_{\mathrm{cover}}(\cX,\|\cdot\|_2,\varepsilon)<\infty\),
\begin{equation}\label{eq:ld-main-disc-resolution}
\operatorname{Disc}_{\mathrm{pm}}
\leq
C\kappa\left\{
\mathsf H_{\cX}(\varepsilon)L
+\frac{\varepsilon^2}{t_{\min}}
\right\}.
\end{equation}
Consequently, if Assumption~\ref{assu:ld-entropy} also holds and the fixed
covering resolution satisfies
\begin{equation}\label{eq:ld-resolution-compatibility}
\varepsilon_0^2\leq t_{\min}k\log k,
\end{equation}
then
\begin{equation}\label{eq:ld-main-disc}
\operatorname{Disc}_{\mathrm{pm}}
\leq
C\kappa k\log k
\left(T+\log\frac1{t_{\min}}\right).
\end{equation}
The constants are independent of \(d,k,N,T,t_{\min}\) and of the data law
beyond the displayed assumptions.
\end{theo}

\begin{proof}
Insert \eqref{eq:ld-anisotropic-complexity} into
\eqref{eq:ld-continuous-reduction} to obtain
\eqref{eq:ld-main-disc-anisotropic}. Alternatively, inserting the Euclidean
specialization \eqref{eq:ld-finite-net-main} and using
\[
1+\int_{t_{\min}}^T\frac{du}{a(u)}
=T+\log\frac1{t_{\min}}=L,
\]
the metric-entropy contribution is at most
\(C\kappa\mathsf H_{\cX}(\varepsilon)L\). For the resolution term,
\eqref{eq:ld-S-bounds} gives
\[
\|S_T\|_2
+\int_{t_{\min}}^T\frac{\|S_u\|_2}{a(u)}\,du
\leq
C\left\{1+\int_{t_{\min}}^1\frac{du}{u^2}
+\int_1^T e^{-m_\bbeta u}\,du\right\}
\leq\frac{C}{t_{\min}}.
\]
This proves \eqref{eq:ld-main-disc-resolution}. Finally, take
\(\varepsilon=\varepsilon_0=k^{-C_0}\). Assumption~\ref{assu:ld-entropy}
controls the first term by \(C\kappa k\log k\,L\), while
\eqref{eq:ld-resolution-compatibility} gives
\[
\frac{\kappa\varepsilon_0^2}{t_{\min}}
\leq C\kappa k\log k\,L
\]
because \(L\geq1\). This proves \eqref{eq:ld-main-disc} directly from
\eqref{eq:ld-main-disc-resolution}.
\end{proof}

\subsection{Overall KL guarantee and step complexity}

Theorem~\ref{theo:ld-cumulative} controls posterior-mean-freezing
discretization. Combining it with the posterior-mean line of
Theorem~\ref{theo:kl-decomposition} and the forward mixing bound in
Lemma~\ref{lem:terminal-mixing} gives the final KL guarantee.

\begin{theo}[Intrinsic-dimensional relative-grid KL guarantee and step complexity]
\label{theo:ld-main}
Suppose Assumptions~\ref{assu:baseline},~\ref{assu:relative-grid},
and~\ref{assu:ld-entropy} hold.
Suppose also that the covering resolution satisfies
\eqref{eq:ld-resolution-compatibility} and that the posterior-mean approximation budget
\eqref{eq:ld-estimation-error} is at most \(\epsilon_{\mathrm{pm}}^2\). Then
\begin{equation}\label{eq:ld-main-kl}
\KL(q_{t_{\min}}\|\widehat p_{T-t_{\min}})
\leq
\epsilon_{\mathrm{pm}}^2+\tau_T
+C\kappa k\log k
\left(T+\log\frac1{t_{\min}}\right).
\end{equation}
If the grid is also rate-regular in the sense of
\eqref{eq:relative-grid-rate}, then
\begin{equation}\label{eq:ld-main-kl-N}
\KL(q_{t_{\min}}\|\widehat p_{T-t_{\min}})
\leq
\epsilon_{\mathrm{pm}}^2+\tau_T
+\frac{Ck\log k}{N}
\left(T+\log\frac1{t_{\min}}\right)^2.
\end{equation}
Under this rate-regular condition, for the overall intrinsic-dimensional
complexity additionally suppose
Assumption~\ref{assu:ld-radius} holds and, for a fixed universal exponent
\(C_d>0\),
\begin{equation}\label{eq:ld-ambient-polynomial-regime}
d\leq k^{C_d}.
\end{equation}
Let \(0<\varepsilon\leq1\), suppose
\(\epsilon_{\mathrm{pm}}^2\leq\varepsilon^2/3\), and take
\[
T=\max\left\{1,\frac1{m_\bbeta}
\log\frac{4(M_2+d)}{\varepsilon^2}\right\}.
\]
For a sufficiently large constant \(C\), the choice
\begin{equation}\label{eq:ld-main-complexity}
\begin{split}
N
&=
\left\lceil
\frac{Ck\log k}{\varepsilon^2}
\left(T+\log\frac1{t_{\min}}\right)^2
\right\rceil\\
&=
O\left(
\frac{k\log k}{\varepsilon^2}
\log^2\frac{k}{\varepsilon}
\right)
\end{split}
\end{equation}
guarantees
\(\KL(q_{t_{\min}}\|\widehat p_{T-t_{\min}})\leq\varepsilon^2\).
\end{theo}

\begin{rema}[Terminal dimension dependence]\label{rema:ld-terminal}
Theorem~\ref{theo:ld-main} replaces \(d\) by \(k\log k\) in the
cumulative discretization term, but not automatically in the standard-Gaussian
initialization term. This is unavoidable at fixed \(T\). For example, when
\(X_0=0\) and \(\bbeta(u)=2I\), the terminal law is
\(N(0,(1-e^{-2T})I)\), whose KL divergence from \(N(0,I)\) is proportional to
\(de^{-4T}\) for large \(T\). The horizon in Theorem~\ref{theo:ld-main}
suppresses this term, so \(d\) enters the step complexity only logarithmically.
For the final line of \eqref{eq:ld-main-complexity},
Assumption~\ref{assu:ld-radius} gives \(M_2\leq R^2=k^{2C_R}\), while
the resolution condition \eqref{eq:ld-resolution-compatibility} gives
\(\log(1/t_{\min})=O(\log k)\), and
\eqref{eq:ld-ambient-polynomial-regime} controls \(d\).
Hence the complexity is \(\widetilde O(k/\varepsilon^2)\).
This treatment of initialization is also used in intrinsic-dimensional DDPM
bounds \cite{huang2026ddpm}. If the sampler is instead initialized from
\(N(0,\Gamma_T)\), convexity of KL gives the covariance-matched estimate
\[
\KL(q_T\|N(0,\Gamma_T))
\leq
\frac12\E\|\Gamma_T^{-1/2}A_TX_0\|^2.
\]
Its dimension dependence is only through the directional second moment of the
data, which is \(O(k)\) when that moment is intrinsically scaled.
\end{rema}

\section{Schedule and grid design from forward corruptions}
\label{sec:adaptive-grids}

Sections~3 and~4 identify the local growth of the two frozen-prediction errors
and control their cumulative values under the forward corruption law. We now
turn those local quantities into design criteria. First, with the instantaneous
noise budget fixed, we choose the direction of the matrix schedule. Second,
with the schedule fixed, we allocate a finite number of reverse steps. This
order separates the geometry of the forward diffusion from its numerical time
resolution. A fixed instantaneous noise budget means a prescribed spectrum at
every forward time, so competing schedules have the same trace and eigenvalues
and differ only in directional allocation.

For score freezing, recall
\[
\mathcal H_s
:=
\E_{X_{T-s}}\!\left[
h_{T-s}(X_{T-s})h_{T-s}(X_{T-s})^T
\right]
\]
and define, for \(0\leq s<t\leq t_N:=T-t_{\min}\),
\begin{equation}\label{eq:adaptive-score-costs}
\mathcal B_{\mathrm{sc}}(s,t)
:=E_{s,t}
=K(t)-K(s)+\Tr(C_{s,t}\mathcal H_s).
\end{equation}
For posterior-mean freezing, put \(u_s=T-s\) and recall
\(\Rbar_u=\E[m_u(X_u)m_u(X_u)^T]\). Theorem~\ref{theo:ld-local-error} gives
\begin{equation}\label{eq:adaptive-pm-costs}
\mathcal B_{\mathrm{pm}}(s,t)
:=\frac12E_{s,t}^{\mathrm{pm}}
=\frac12\Tr\left[W_{T-t}\{\Rbar_{T-t}-\Rbar_{u_s}\}\right].
\end{equation}
The factor \(1/2\) in \(\mathcal B_{\mathrm{pm}}\) is the Girsanov factor. For
either \(\star\in\{\mathrm{sc},\mathrm{pm}\}\), let
\begin{align}
\mathcal C_\star[s,b]
&:=\int_s^b\mathcal B_\star(s,t)\,dt,
\label{eq:adaptive-interval-cost}\\
g_\star(s)
&:=\left.\partial_t^+\mathcal B_\star(s,t)\right|_{t=s}.
\label{eq:adaptive-local-growth}
\end{align}
Positive forward time and Assumption~\ref{assu:baseline} justify the right
derivative. At schedule breakpoints it is taken from the side \(t>s\). Writing
\(u=T-s\), direct differentiation gives
\begin{equation}\label{eq:adaptive-growth-formulas}
\begin{aligned}
g_{\mathrm{sc}}(s)
&=K'(s)+\Tr\!\left[
\{\dot\bbeta(u)+\bbeta(u)^2\}\mathcal H_s
\right],\\
g_{\mathrm{pm}}(s)
&=\frac12\Tr\{W_u(-\Rbar_u')\}.
\end{aligned}
\end{equation}
Here \(\dot\bbeta(u)=d\bbeta(u)/du\). In either branch,
\begin{equation}\label{eq:adaptive-cost-growth-equivalence}
\mathcal C_\star[s,s+\Delta]
=\frac12g_\star(s)\Delta^2+o(\Delta^2).
\end{equation}
Thus \(g_\star\) is the common infinitesimal cost used below.

\subsection{Schedule alignment under a fixed noise budget}

At each forward time, prescribe a positive diagonal matrix \(\Lambda(u)\) and
consider
\begin{equation}\label{eq:fixed-spectrum-class}
\mathcal A_u
:=\{U\Lambda(u)U^T:U\in O(d)\}.
\end{equation}
Schedules in \(\mathcal A_u\) have the same instantaneous spectrum, trace, and
contraction bounds; only their directions differ. The alignment criterion
depends on the frozen prediction target.

\paragraph{Score alignment.}
For a symmetric matrix \(A\), define
\[
\mathscr K_A(u):=\E[h_u(X_u)^TAh_u(X_u)].
\]
Thus \(\mathscr K_{\bbeta(u)}(u)=K(T-u)\). For \(A\succeq0\), define the
weighted score-Hessian energy
\begin{equation}\label{eq:score-hessian-energy}
\widetilde J_A(u)
:=\int q_u(x)
\left\|A^{1/2}\nabla^2\log q_u(x)A^{1/2}\right\|_F^2dx.
\end{equation}

\begin{prop}[Forward score-curvature identity]
\label{prop:forward-score-curvature}
On each smooth interval of \(\bbeta\),
\begin{equation}\label{eq:ordinary-fisher-evolution}
\frac{d}{du}\mathscr K_{\bbeta(u)}(u)
=\mathscr K_{\dot\bbeta(u)+\bbeta(u)^2}(u)
-\widetilde J_{\bbeta(u)}(u).
\end{equation}
Consequently, for \(u=T-s\),
\begin{equation}\label{eq:score-growth-curvature}
g_{\mathrm{sc}}(s)=\widetilde J_{\bbeta(u)}(u).
\end{equation}
\end{prop}

The proof is given in Appendix~\ref{app:adaptive-grid-proof}. The cancellation
of \(\dot\bbeta\) shows that, conditional on the current marginal, the leading
local error depends on the schedule direction but not its rotation speed.
Assigning a large noising rate to a direction of large score curvature is
therefore expensive. Given the current forward marginal, the local
fixed-spectrum oracle is
\begin{equation}\label{eq:local-schedule-oracle}
B_{\mathrm{or}}(u)
\in\operatorname*{argmin}_{B\in\mathcal A_u}\widetilde J_B(u).
\end{equation}
For a two-dimensional block this is a one-dimensional angle search. The
identity remains exact for rotating schedules; the corresponding cumulative
worst-case bound is stated in Remark~\ref{rema:nonmonotone-schedules}.

The Gaussian case gives an explicit alignment rule.
\begin{coro}[Gaussian fixed-spectrum alignment]
\label{prop:gaussian-schedule-diagnostic}
Suppose \(q_u=N(\mu_u,V_u)\), where \(V_u\succ0\). Then, at \(u=T-s\),
\begin{equation}\label{eq:gaussian-schedule-growth}
g_{\mathrm{sc}}(s)
=\Tr\!\left[\bbeta(u)V_u^{-1}\bbeta(u)V_u^{-1}\right]
=\left\|V_u^{-1/2}\bbeta(u)V_u^{-1/2}\right\|_F^2.
\end{equation}
For a two-dimensional fixed spectrum, a minimizer shares the eigenvectors of
\(V_u\) and pairs the larger rate with the larger variance.
\end{coro}

The proof is given in Appendix~\ref{app:adaptive-grid-proof}.

\paragraph{Intrinsic alignment.}
For posterior-mean freezing, schedule directions enter the cumulative bound
through the anisotropic entropy \(\mathfrak C_{\bbeta}(u)\) in
\eqref{eq:ld-anisotropic-entropy}. The next result covers a fixed linear
subspace.

\begin{prop}[Intrinsic subspace alignment]
\label{prop:intrinsic-schedule-alignment}
Let \(Q=(q_1,\ldots,q_d)\in O(d)\) and suppose
\[
\cX\subset x_\star+\mathcal U,
\qquad
\mathcal U:=\operatorname{span}\{q_1,\ldots,q_k\}.
\]
Let \(b_1,\ldots,b_d\) be positive rate curves whose cumulative rates
\(\Lambda_j(u):=\int_0^u b_j(v)\,dv\) satisfy
\(\Lambda_1(u)\geq\cdots\geq\Lambda_d(u)\) for every \(u>0\). For a
permutation \(\sigma\), consider the commuting fixed-spectrum schedule
\begin{equation}\label{eq:intrinsic-permuted-schedule}
\bbeta_\sigma(u)
:=Q\operatorname{diag}\{b_{\sigma(1)}(u),\ldots,b_{\sigma(d)}(u)\}Q^T.
\end{equation}
Assume these schedules satisfy Assumption~\ref{assu:baseline} with common
constants. For every \(\sigma\), there exists a permutation
\(\sigma^\uparrow\) that assigns the \(k\) largest cumulative rates to
\(\mathcal U\), namely
\(\{\sigma^\uparrow(1),\ldots,\sigma^\uparrow(k)\}=\{1,\ldots,k\}\), and
satisfies, for every \(u>0\) and \(x,y\in\cX\),
\begin{equation}\label{eq:intrinsic-alignment-metric}
d_{\bbeta_{\sigma^\uparrow},u}(x,y)
\leq d_{\bbeta_\sigma,u}(x,y).
\end{equation}
Consequently,
\begin{equation}\label{eq:intrinsic-alignment-complexity}
\mathfrak C_{\bbeta_{\sigma^\uparrow}}(u)
\leq\mathfrak C_{\bbeta_\sigma}(u),
\end{equation}
and the schedule-dependent cumulative upper bound in
\eqref{eq:ld-main-disc-anisotropic} cannot increase on any common grid
satisfying Assumption~\ref{assu:relative-grid}.
\end{prop}

The proof is given in Appendix~\ref{app:adaptive-grid-proof}. It shows that the
largest cumulative rates should be assigned to the data subspace. Unlike the
score criterion, this proposition compares the intrinsic-dimensional upper
bounds rather than the exact local error, and it does not claim optimality for
general manifolds or noncommuting schedules.

\paragraph{Pilot rule and overall objective.}
The score oracle in \eqref{eq:local-schedule-oracle} is circular because
\(q_u\) depends on the preceding schedule. We therefore generate a reference
forward path \(q_u^{\mathrm{ref}}\), for example under isotropic noising, and
set

\begin{equation}\label{eq:pilot-schedule}
B_{\mathrm{pil}}(u)
\in\operatorname*{argmin}_{B\in\mathcal A_u}
\widetilde J_B^{\mathrm{ref}}(u).
\end{equation}
Here the superscript identifies the marginal used in the expectation. For a
block rotation, the objective can be evaluated analytically or estimated from
score-Hessian evaluations on forward-corrupted pilot samples; the selected
directions are then smoothly interpolated in time. This is the rule used in
Section~6.

For a schedule \(\bbeta\), define its score growth length
\begin{equation}\label{eq:score-growth-length}
L_{\mathrm{sc}}(\bbeta)
:=\int_{t_{\min}}^T
\sqrt{\widetilde J_{\bbeta(u)}^{\bbeta}(u)}\,du,
\end{equation}
where the superscript means that the marginal is generated by \(\bbeta\).
After optimizing the grid, the leading discretization cost is
\(L_{\mathrm{sc}}(\bbeta)^2/(2N)\). This suggests the forward design objective
\begin{equation}\label{eq:schedule-design-objective}
\mathcal J_N(\bbeta)
:=\KL(q_T^{\bbeta}\|\pi)
+\frac1{2N}
\left\{
\int_{t_{\min}}^T
\sqrt{\widetilde J_{\bbeta(u)}^{\bbeta}(u)}\,du
\right\}^2.
\end{equation}
The two terms balance terminal Gaussianization and discretization. This is a
leading-order design objective, not a global optimality theorem; the pilot rule
is a tractable surrogate for its second term. A learned predictor also adds the
schedule-dependent approximation budget from Section~2. For a pretrained model
whose schedule cannot be changed, only the grid-design step below applies.

\subsection{Grid design for a fixed schedule}

We now hold \(\bbeta\) fixed and allocate \(N\) reverse steps. For a fine grid
with step lengths \(\Delta_i\), \eqref{eq:adaptive-cost-growth-equivalence}
and Cauchy--Schwarz give
\[
\frac12\sum_{i=0}^{N-1}g_\star(t_i)\Delta_i^2
\geq\frac1{2N}
\left\{\sum_{i=0}^{N-1}\sqrt{g_\star(t_i)}\,\Delta_i\right\}^2.
\]
Thus the leading interval costs are equalized by making
\(\Delta_i\sqrt{g_\star(t_i)}\) nearly constant.

\begin{defi}[Adaptive grid from local error]
\label{def:adaptive-grid}
Choose \(\star\in\{\mathrm{sc},\mathrm{pm}\}\). The adaptive grid
\(0=t_0<\cdots<t_N=T-t_{\min}\) satisfies
\begin{equation}\label{eq:adaptive-square-root-grid}
\int_{t_i}^{t_{i+1}}\sqrt{g_\star(r)}\,dr
=\frac{L_\star}{N},
\qquad i=0,\ldots,N-1,
\end{equation}
where \(L_\star:=\int_0^{t_N}\sqrt{g_\star(r)}\,dr\). Equivalently, the
interior endpoints are quantiles of \(\sqrt{g_\star(r)}\,dr\), chosen by its
generalized inverse on flat regions.
\end{defi}

For score freezing, Proposition~\ref{prop:forward-score-curvature} gives
the score growth length in \eqref{eq:score-growth-length}, which is the
quantity used in \eqref{eq:schedule-design-objective}. For posterior-mean freezing,
\eqref{eq:adaptive-growth-formulas} instead uses increments of the
unconditional clean-prediction energy \(\Rbar_u\).

The nonnegative rate \(g_\star\) may vanish on locally flat regions. For a
strictly increasing numerical cumulative function, one may use
\(g_{\star,\rho}=g_\star+\rho\), \(\rho>0\). Equivalently,
\(\mathcal C_{\star,\rho}[s,b]
=\mathcal C_\star[s,b]+\rho(b-s)^2/2\). The leading constant converges to the
unregularized value as \(\rho\downarrow0\).

\begin{prop}[Asymptotic optimality of the adaptive grid]
\label{prop:adaptive-asymptotic}
Suppose that \(g_\star\) is piecewise continuous and
\(g_\star(s)\geq g_0>0\). For a regularized grid, interpret
\((\mathcal C_\star,g_\star)\) below as
\((\mathcal C_{\star,\rho},g_{\star,\rho})\). Then
\begin{equation}\label{eq:adaptive-grid-optimal-value}
\sum_{i=0}^{N-1}\mathcal C_\star[t_i,t_{i+1}]
=\frac{L_\star^2}{2N}+o(N^{-1}).
\end{equation}
Moreover, for every grid sequence
\(0=t_0'<\cdots<t_N'=t_N\) with mesh \(O(N^{-1})\),
\begin{equation}\label{eq:adaptive-grid-lower-bound}
\liminf_{N\to\infty}
N\sum_{i=0}^{N-1}\mathcal C_\star[t_i',t_{i+1}']
\geq\frac12L_\star^2.
\end{equation}
\end{prop}

\begin{prop}[Dimension dependence]
\label{prop:adaptive-dimension}
Suppose the unregularized cost satisfies the regularity conditions of
Proposition~\ref{prop:adaptive-asymptotic}. Under the assumptions of Section~3,
\[
L_{\mathrm{sc}}^2\leq
Cd\left(T+\log\frac1{t_{\min}}\right)^2.
\]
Under the assumptions of Section~4,
\[
L_{\mathrm{pm}}^2\leq
Ck\log k\left(T+\log\frac1{t_{\min}}\right)^2.
\]
\end{prop}

The proofs are given in Appendix~\ref{app:adaptive-grid-proof}. Once
\(g_\star\) is available, the grid is computed from
\begin{equation}\label{eq:adaptive-cumulative-computation}
F_\star(t):=\int_0^t\sqrt{g_\star(r)}\,dr,
\qquad
t_i=F_\star^{-1}\!\left(\frac{i}{N}F_\star(t_N)\right).
\end{equation}
In tractable simulations, \(g_\star\) is evaluated from the known forward
marginals, followed by one-dimensional quadrature and inversion of
\(F_\star\). This is the procedure used for the Gaussian mixture in Section~6.

For a trained diffusion model, the exact growth is generally unavailable. One
can instead forward-noise a small pilot batch on a lattice
\(0=r_0<\cdots<r_M=t_N\). For the common scalar schedule
\(\bbeta(u)=b(u)I\), it is more stable to estimate the cell integral of
\(g_{\mathrm{sc}}\) than to differentiate a noisy estimate of \(K\). With
\(\Delta_\ell=r_{\ell+1}-r_\ell\), set
\begin{equation}\label{eq:pretrained-growth-estimator}
\begin{aligned}
\widehat K(r_\ell)
&:=\frac1{B_{\mathrm{pil}}}\sum_{j=1}^{B_{\mathrm{pil}}}
\left\|\widehat h_{u_\ell}(X_{u_\ell}^{(j)})\right\|_{\bbeta(u_\ell)}^2,
\qquad u_\ell=T-r_\ell,\\
\widehat G_\ell
&:=\widehat K(r_{\ell+1})-\widehat K(r_\ell)
+\int_{r_\ell}^{r_{\ell+1}}
\lambda_{\mathrm{sch}}(r)\widehat K(r)\,dr,\\
\widehat g_\ell
&:=\max\left\{\frac{\widehat G_\ell}{\Delta_\ell},\rho\right\},
\qquad
\lambda_{\mathrm{sch}}(r)
:=\frac{\dot b(T-r)}{b(T-r)}+b(T-r).
\end{aligned}
\end{equation}
Here \(B_{\mathrm{pil}}\) is the pilot batch size and \(\rho>0\) is a small
floor. Treating \(\widehat g_\ell\) as constant on each pilot cell gives a
numerical approximation to \(F_\star\). For posterior-mean freezing, one
instead estimates \(\Rbar_u\) from clean-data predictions and approximates
\(-\Rbar_u'\) by pilot-cell differences in
\eqref{eq:adaptive-growth-formulas}. Both constructions use only forward
corruptions and require no retraining.

\section{Experiments}\label{sec:numerical-examples}

We evaluate the schedule and grid rules of Section~5 on a high-dimensional
Gaussian mixture. This controlled experiment is designed to isolate the
discretization term: its anisotropy changes with the forward noise level, so it
distinguishes fixed and rotating matrix schedules, while its product structure
permits accurate evaluation of the local error. The experiment is not intended
as a substitute for image-scale validation, where estimating dense directional
criteria and retraining under a changed schedule introduce additional costs.

For a grid \(\mathcal T_N\), let
\[
\mathcal E_n:=\mathcal C_{\mathrm{sc}}[t_{n-1},t_n],
\qquad
\mathrm{Disc}_N:=\sum_{n=1}^N\mathcal E_n.
\]
Thus \(\mathrm{Disc}_N\) is the exact-score discretization term in
Theorem~\ref{theo:kl-decomposition}. We report it directly to isolate schedule
and grid effects; terminal initialization is negligible in this example.

Let \(T=10\), \(t_{\min}=0.02\), and \(q_0=\nu_0^{\otimes50}\), so that
\(d=100\), where each two-dimensional factor is
\begin{equation}\label{eq:rotating-schedule-mixture}
\nu_0=\frac14\sum_{z\in\{-1,1\}^2}
\mathcal N\left(R_{75^\circ}
\begin{pmatrix}2&0\\0&8\end{pmatrix}z,
\begin{pmatrix}0.05&0\\0&5\end{pmatrix}\right).
\end{equation}
The within-component covariance is diagonal in the original basis, whereas
the component means vary along directions rotated by \(75^\circ\). These two
geometries dominate at different forward-noise levels. The leading direction
of the marginal covariance corresponds to a rate-axis angle
\(\theta_{\rm mar}=73.73^\circ\).

For a two-dimensional block, put
\[
B_\theta=R_\theta
\begin{pmatrix}r_-&0\\0&r_+\end{pmatrix}R_\theta^T,
\qquad
\frac{r_++r_-}{2}=0.85,
\qquad
\frac{r_+}{r_-}=1.3,
\]
and set \(\bbeta_\theta=I_{50}\otimes B_\theta\). We compare two fixed
schedules, with \(\theta=0\) and \(\theta=\theta_{\rm mar}\), and two rotating
schedules. To locate the change of direction from forward-marginal curvature, let
\(q_u^{\rm iso}\) be the isotropically noised marginal and consider
\begin{equation}\label{eq:forward-curvature-direction}
\theta_{\rm pil}(u)
\in\operatorname*{argmin}_{0\leq\theta\leq\theta_{\rm mar}}
\E_{q_u^{\rm iso}}
\left\|B_\theta^{1/2}
\nabla^2\log q_u^{\rm iso}(X_u)
B_\theta^{1/2}\right\|_F^2.
\end{equation}
This specializes the fixed-spectrum pilot rule \eqref{eq:pilot-schedule} to a
one-angle block family. Its minimizer changes from the within-component
direction to the marginal direction near \(u=4.4\). We therefore use
\[
\theta_{\rm fwd}(u)=\theta_{\rm mar}
\frac{\ell(u)-\ell(0)}{\ell(T)-\ell(0)},
\qquad
\ell(u)=\{1+e^{-(u-4.4)/0.3}\}^{-1},
\qquad
\theta_{\rm rev}(u)=\theta_{\rm mar}-\theta_{\rm fwd}(u).
\]
All four schedules have the same trace and instantaneous spectrum. The
rotating schedules are noncommutative and differ only in the order in which
they follow the two geometries. Their comparison therefore isolates
directional design rather than the total amount of noise.

For each schedule, we compare uniform, hybrid, and adaptive grids at
\(N\in\{16,32,64,96\}\). The adaptive grid uses the square-root quantiles of
\(g_{\mathrm{sc}}\), evaluated on a 240-cell pilot grid. By the product
structure, all forward expectations and interval errors are sums of 50
two-dimensional terms. We evaluate each term by 28-point Gauss--Hermite
quadrature in each coordinate; increasing the order to 44 leaves the reported
values unchanged.

\clearpage
\begin{figure}[!t]
\centering
\includegraphics[width=\textwidth]{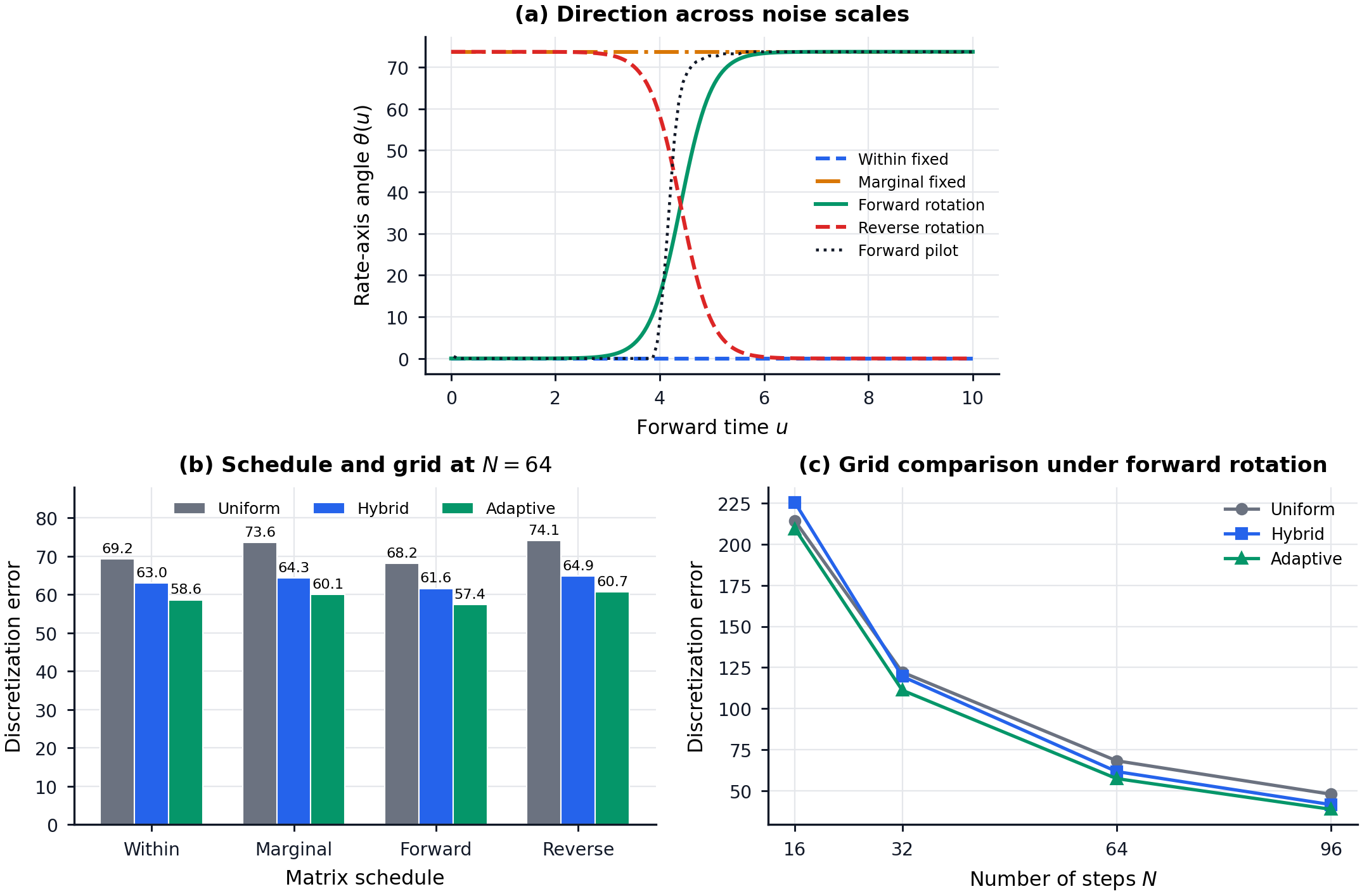}
\caption{Matrix schedule and grid comparison on the product mixture
\eqref{eq:rotating-schedule-mixture}. Panel (a) shows the four schedule
directions and the forward-pilot direction from
\eqref{eq:forward-curvature-direction}. Panel (b) compares every schedule-grid
pair at \(N=64\). Panel (c) compares the three grids for forward rotation, and
shows how their errors change with the number of steps.}
\label{fig:matrix-schedule-grid-comparison}
\end{figure}

Figure~\ref{fig:matrix-schedule-grid-comparison} gives two consistent findings.
\begin{itemize}
\item \textbf{Schedule direction.} The pilot direction in panel (a) moves from
the within-component geometry to the marginal geometry. Forward rotation
follows this order, whereas reverse rotation traverses it backward. On the
adaptive grids at \(N=64\), the errors are \(58.65\), \(60.08\), \(57.36\),
and \(60.66\) for the within-fixed, marginal-fixed, forward-rotation, and
reverse-rotation schedules. Forward rotation is best in this comparison. Thus
rotation alone is insufficient; its ordering must track the scale-dependent
forward geometry.
\item \textbf{Grid and step count.} At \(N=64\), the adaptive grid has the
smallest error for every schedule. Under forward rotation, the uniform,
hybrid, and adaptive errors are \(68.20\), \(61.64\),
and \(57.36\), so adaptation improves on the two baselines by \(15.9\%\) and
\(6.9\%\). Panel (c) shows that this ordering persists from \(N=16\) to
\(N=96\), where the corresponding errors are \(47.84\), \(41.52\), and
\(38.66\).
\end{itemize}
Across all schedules, the terminal KL lies between
\(5.0\times10^{-4}\) and \(1.7\times10^{-2}\), confirming that these
comparisons are governed by discretization.

\section{Conclusion}

We developed a unified forward-evolution framework for reverse-SDE
discretization with matrix-valued schedules. For both score and posterior-mean
freezing, the argument identifies the local prediction-freezing error,
accumulates it by Stieltjes integration by parts on the forward time axis, and
then inserts the cumulative estimate into a common KL decomposition. The local
objects differ across the two schemes. Weighted score energy and
Fisher-information dissipation give the ambient-dimensional result, while
clean-prediction covariance increments and metric entropy give the
intrinsic-dimensional result. Thus pointwise equivalence of the score and
posterior mean does not make their frozen numerical schemes equivalent.

The same local objects connect error analysis with design. Weighted
score-Hessian energy gives a directional criterion for a matrix schedule, and
the posterior-mean bound gives a subspace-alignment rule. Once the schedule is
fixed, square-root equidistribution of the local growth produces an
asymptotically optimal grid for the leading discretization cost. The
Gaussian-mixture experiment illustrates both effects: forward-aligned rotation
improves the schedule comparison at the reference step count, and adaptive
grids reduce discretization across several step counts.

The results match rather than improve the best known worst-case dimension
orders. Their contribution is the unified forward proof, its matrix-sensitive
local description, and the schedule and grid criteria obtained from that
description. The intrinsic guarantee applies to discretization; initialization
from a standard Gaussian can still depend on the ambient dimension at fixed
terminal time.

The present experiment is deliberately controlled. Dense matrix schedules and
score-curvature estimation may be expensive for high-resolution data, and a
changed forward schedule can require retraining the predictor. Practical
structured parameterizations, statistical guarantees for pilot estimates, and
image-scale validation are therefore important next steps. The broader idea of
analyzing reverse errors through a prescribed forward evolution may also extend
to more general stochastic interpolants.

\appendix

\section{Cumulative estimates}\label{app:cumulative-auxiliary}

This appendix collects the estimates used in the cumulative argument of Section~3. They are standard consequences of Gaussian smoothing, the finite second moment assumption, and boundedness of the schedule.

\begin{lemm}[Forward Gaussian smoothing]\label{lem:gaussian-smoothing-score}
Let
$$
X_u=\Phi(u,0)X_0+Z_u,
$$
where \(Z_u\) is centered Gaussian and independent of \(X_0\), with \(\operatorname{Cov}(Z_u)=\Gamma_u\). Then
$$
h_u(X_u)=\E_{Z_u\mid X_u}[-\Gamma_u^{-1}Z_u],
\qquad
\Gamma_u=I-\Phi(u,0)\Phi(u,0)^T.
$$
Moreover, if \(m_\bbeta I\preceq\bbeta(u)\) on \([0,u]\), then
$$
\Gamma_u\succeq(1-e^{-m_\bbeta u})I.
$$
\end{lemm}

\begin{proof}
The density of \(X_u\) is the Gaussian convolution of the law of \(\Phi(u,0)X_0\) with \(N(0,\Gamma_u)\). Differentiating this convolution in \(x\) gives
$$
h_u(X_u)=\E_{Z_u\mid X_u}[-\Gamma_u^{-1}Z_u].
$$
The covariance identity follows from the covariance equation
$$
\Gamma_u'=-\frac12\bbeta(u)\Gamma_u-\frac12\Gamma_u\bbeta(u)+\bbeta(u),
\qquad \Gamma_0=0,
$$
which is also solved by \(I-\Phi(u,0)\Phi(u,0)^T\). For any vector \(w\),
\[
\frac{d}{dr}\|\Phi(r,0)w\|^2
=-(\Phi(r,0)w)^T\bbeta(r)\Phi(r,0)w
\leq -m_\bbeta\|\Phi(r,0)w\|^2.
\]
Therefore \(\|\Phi(u,0)\|_2^2\leq e^{-m_\bbeta u}\). Since \(\|\Phi(u,0)^T\|_2=\|\Phi(u,0)\|_2\), for every vector \(v\),
\[
v^T\Gamma_uv
=\|v\|^2-\|\Phi(u,0)^Tv\|^2
\geq(1-e^{-m_\bbeta u})\|v\|^2,
\]
which proves the lower bound on \(\Gamma_u\).
\end{proof}

\begin{proof}[Proof of Lemma~\ref{lem:cumulative-technical-bounds}]
Jensen and Lemma~\ref{lem:gaussian-smoothing-score} give
$\E\|h_u(X_u)\|^2\leq d/(1-e^{-m_\bbeta u})$.
Also, $\|\Phi(u,0)\|_2\leq1$ and $\Gamma_u\preceq I$ imply
$\E\|X_u\|^2\leq M_2+d$.
Since $\zeta_u(X_u)=h_u(X_u)+X_u$,
\[
I_\bbeta(u)
\leq2B_\bbeta\left(M_2+d+\frac{d}{1-e^{-m_\bbeta u}}\right).
\]
The elementary inequality
$(1-e^{-x})^{-1}\leq1+x^{-1}$, together with
$M_2\leq C_{\mathrm{mom}}d$,
proves the first displayed bound. The same score and moment
estimates imply the stated bound on $K(T-u)$.

For the covariance term, \(\Sigma_u\) satisfies
$$
\Sigma_u'=-\frac12\bbeta(u)\Sigma_u-\frac12\Sigma_u\bbeta(u)+\bbeta(u),
$$
and the previous moment bound gives \(\Tr(\Sigma_u)\leq Cd\). Using \(\|\bbeta(u)\|_2\leq B_\bbeta\) and \(\|\bbeta'(u)\|_2\leq L_\bbeta\), we obtain \(|G'(u)|\leq Cd\).
The functions \(\bbeta\) and \(\Sigma\) are continuous, so \(G\) is continuous. For \(v\leq u\), partitioning \([v,u]\subseteq[t_{\min},T]\) at the finitely many schedule breakpoints and integrating the derivative bound on each smooth piece gives
\[
|G(u)-G(v)|\leq Cd|u-v|.
\]
The case \(u<v\) follows by symmetry.

For the matrix estimate, set
\[
\nu:=T-t,\qquad \mu:=T-s,\qquad h:=\mu-\nu=t-s,
\qquad \Phi:=\Phi(\mu,\nu).
\]
For every vector \(v\) and \(r\in[\nu,\mu]\),
\[
\frac{d}{dr}\|\Phi(r,\nu)v\|^2
=-\left(\Phi(r,\nu)v\right)^T
\bbeta(r)\Phi(r,\nu)v
\leq0.
\]
Thus
\[
\|\Phi(r,\nu)\|_2\leq1,
\qquad r\in[\nu,\mu].
\]
Integrating the transition equation gives
\[
I-\Phi
=\frac12\int_\nu^\mu\bbeta(r)\Phi(r,\nu)\,dr,
\]
and hence
\[
\|I-\Phi\|_2
\leq
\frac12\int_\nu^\mu
\|\bbeta(r)\|_2\|\Phi(r,\nu)\|_2\,dr
\leq \frac12B_\bbeta h.
\]

The interval \([\nu,\mu]\) may cross breakpoints of the piecewise-\(C^1\) schedule. Partition it at those finitely many breakpoints, writing
\[
\nu=r_0<r_1<\cdots<r_J=\mu.
\]
Continuity of \(\bbeta\), the fundamental theorem of calculus on each smooth piece, and the derivative bound give
\[
\begin{aligned}
\|\bbeta(\mu)-\bbeta(\nu)\|_2
&\leq
\sum_{j=1}^J
\|\bbeta(r_j)-\bbeta(r_{j-1})\|_2\\
&\leq
\sum_{j=1}^J\int_{r_{j-1}}^{r_j}
\|\bbeta'(r)\|_2\,dr
\leq L_\bbeta h.
\end{aligned}
\]
This proves the required schedule-variation estimate even when \([\nu,\mu]\) crosses a breakpoint.

Using the exact decomposition
\[
C_{s,t}
=\bbeta(\mu)-\bbeta(\nu)
+(I-\Phi)\bbeta(\nu)
+\bbeta(\nu)(I-\Phi^T),
\]
we obtain
\[
\begin{aligned}
\|C_{s,t}\|_2
&\leq
\|\bbeta(\mu)-\bbeta(\nu)\|_2
+2\|I-\Phi\|_2\|\bbeta(\nu)\|_2\\
&\leq
\left(L_\bbeta+B_\bbeta^2\right)h.
\end{aligned}
\]
Moreover, \(C_{s,t}\) is symmetric because \(\bbeta(\mu)\) and \(\bbeta(\nu)\) are symmetric and
\[
\Phi\bbeta(\nu)+\bbeta(\nu)\Phi^T
\]
is the sum of a matrix and its transpose. Consequently, \(M_{s,t}\) is also symmetric. Since \(\bbeta(\mu)\succeq m_\bbeta I\),
\[
\begin{aligned}
\|M_{s,t}\|_2
&\leq
\|\bbeta(\mu)^{-1/2}\|_2^2
\|C_{s,t}\|_2\\
&\leq
\frac{L_\bbeta+B_\bbeta^2}{m_\bbeta}(t-s).
\end{aligned}
\]
Finally, symmetry gives
\[
\lambda_{\max}(M_{s,t})
\leq \|M_{s,t}\|_2,
\]
which proves both asserted estimates without any commutativity assumption.
\end{proof}

\begin{proof}[Proof of Theorem~\ref{theo:relative-grid-kl-complexity}]
The KL decomposition in Theorem~\ref{theo:kl-decomposition}, the cumulative
bound in Theorem~\ref{theo:cumulative-E-relative-grid}, and
Lemma~\ref{lem:terminal-mixing} give
\eqref{eq:relative-grid-score-error}. If \eqref{eq:relative-grid-rate} also
holds, then
\[
\operatorname{Disc}_{\mathrm{score}}(\mathcal T_N)
\leq\frac{Cd}{N}
\left(T+\log\frac1{t_{\min}}\right)^2.
\]
This proves \eqref{eq:rate-regular-final-error}.
For the stated terminal horizon,
\[
e^{-\Lambda_T}\leq e^{-m_\bbeta T}
\leq\frac{\varepsilon^2}{4(M_2+d)}\leq\frac14.
\]
Writing \(S=M_2+d\), substitution into
\eqref{eq:terminal-mixing-rate} gives
\[
\tau_T
\leq\frac{M_2\varepsilon^2}{8S}
+\frac{d\varepsilon^4}{48S^2}
\leq\frac{7}{48}\varepsilon^2
<\frac13\varepsilon^2.
\]
The score term is at most
\(\varepsilon^2/3\) by assumption, and a sufficiently large \(C\) in
\eqref{eq:required-N-rate-regular} makes the discretization term at most
\(\varepsilon^2/3\). This proves the claimed KL guarantee.
Finally, \(M_2\leq C_{\mathrm{mom}}d\) implies
\[
T+\log\frac1{t_{\min}}
\leq C\left(1+\log\frac{d}{\varepsilon^2}
+\log\frac1{t_{\min}}\right),
\]
which proves the second line of \eqref{eq:required-N-rate-regular}.

To verify the Wasserstein statement following the theorem, couple \(q_0\) and
\(q_u\) through \(X_u=A_uX_0+\Gamma_u^{1/2}Z\). The forward equation and
Assumption~\ref{assu:baseline} give \(\|A_u-I\|_2\leq Cu\) and
\(\Tr(\Gamma_u)\leq Cdu\) for \(0\leq u\leq1\). Hence
\[
W_2^2(q_0,q_u)
\leq \E\|X_u-X_0\|^2
\leq \|A_u-I\|_2^2M_2+\Tr(\Gamma_u)
\leq Cdu.
\]
Taking \(u=t_{\min}\lesssim\delta^2/d\) proves the claim.
\end{proof}

\begin{rema}[Nonmonotone schedules]\label{rema:nonmonotone-schedules}
If \(\bbeta'\) is not negative semidefinite, one may replace the monotonicity of \(I_\bbeta\) by the integrating factor
$$
F(u):=e^{-A(u)}I_\bbeta(u),
\qquad
A(u):=\int_{t_{\min}}^u
\left[
\lambda_{\max}\left(\bbeta(r)^{-1/2}\bbeta'(r)\bbeta(r)^{-1/2}\right)
\right]_+dr.
$$
Then \(F\) is nonincreasing. Repeating the proof of Theorem~\ref{theo:cumulative-E-relative-grid} gives the same bound with an additional multiplier \(e^{A_\bbeta}\), where \(A_\bbeta=A(T)\). We keep the monotone version in the main text because it gives the clearest comparison with the scalar discretization bound \citep[Theorem~1]{benton2023nearly}.
\end{rema}

\section{Terminal mixing and initialization}\label{app:terminal-initialization}

Set
\[
A_T:=\Phi(T,0),\qquad
\Gamma_T:=I-A_TA_T^T,
\qquad
\rho_T:=\|A_T\|_2^2,
\qquad
\Sigma_0:=\E[X_0X_0^T],
\]
and define the directional intermediate bound
\[
\tau_T^{\mathrm{dir}}
:=\frac12\Tr(A_T\Sigma_0A_T^T)
+\frac12\{\Tr(\Gamma_T)-\log\det\Gamma_T-d\}.
\]
The proof below shows
\[
\KL(q_T\|\pi)
\leq\tau_T^{\mathrm{dir}}
\leq\frac{M_2}{2}\rho_T
+\frac{d}{4}\frac{\rho_T^2}{1-\rho_T}
\leq\tau_T.
\]

\begin{proof}[Proof of Lemma~\ref{lem:terminal-mixing}]
Write the forward solution as
\[
X_T=A_TX_0+Z_T,
\]
where $Z_T$ is centered Gaussian, independent of $X_0$, with covariance
$\Gamma_T$. Convexity of relative entropy and the Gaussian KL formula give
\begin{align*}
\KL(q_T\|\pi)
&\leq\E_{X_0}\KL\bigl(N(A_TX_0,\Gamma_T)\|N(0,I)\bigr)\\
&=\frac12\Tr(A_T\Sigma_0A_T^T)
+\frac12\{\Tr(\Gamma_T)-\log\det\Gamma_T-d\}
=\tau_T^{\mathrm{dir}}.
\end{align*}

Let $\lambda_1,\ldots,\lambda_d$ be the eigenvalues of $A_TA_T^T$.
Since $0\leq\lambda_i\leq\rho_T<1$ and
\[
-a-\log(1-a)
=\int_0^a\frac{s}{1-s}\,ds
\leq\frac{a^2}{2(1-a)},
\qquad 0\leq a<1,
\]
we have
\[
\Tr(\Gamma_T)-\log\det\Gamma_T-d
\leq\frac{d\rho_T^2}{2(1-\rho_T)}.
\]
Moreover,
$\Tr(A_T\Sigma_0A_T^T)\leq\rho_T\Tr(\Sigma_0)=\rho_TM_2$.
This proves the middle inequality in the theorem.

Finally, for every vector $v$,
\[
\frac{d}{dt}\|\Phi(t,0)v\|^2
=-(\Phi(t,0)v)^T\bbeta(t)\Phi(t,0)v
\leq-\lambda_{\min}(\bbeta(t))\|\Phi(t,0)v\|^2.
\]
Integration yields $\rho_T\leq e^{-\Lambda_T}$. Since
$r\mapsto r^2/(1-r)$ is increasing on $[0,1)$, substituting this estimate
gives the bound by $\tau_T$.
\end{proof}

\section{Intrinsic-dimensional proofs}
\label{app:intrinsic-proofs}

\subsection{Posterior-mean identities}

\begin{proof}[Proof of the matrix identities in \eqref{eq:ld-tweedie}]
Conditionally on \(X_0=x_0\), the random variable \(X_u\) is Gaussian with mean \(A_ux_0\) and covariance \(\Gamma_u\). Differentiating the Gaussian mixture density gives
\[
\nabla q_u(x)
=
\int
-\Gamma_u^{-1}(x-A_ux_0)
\varphi_{\Gamma_u}(x-A_ux_0)\,q_0(dx_0).
\]
After division by \(q_u(x)\), the conditional expectation of \(X_0\) given \(X_u=x\) yields the first identity in \eqref{eq:ld-tweedie}. Substitution gives the posterior-mean representation of the reverse drift in \eqref{eq:true-reverse}.

It remains to verify \eqref{eq:ld-S-W}. Let \(B_u:=A_u^TA_u\). The push-through identity gives
\[
S_u=A_u^T(I-A_uA_u^T)^{-1}A_u=(I-B_u)^{-1}-I.
\]
Since \(A_u'=-\bbeta(u)A_u/2\),
\[
B_u'=-A_u^T\bbeta(u)A_u.
\]
Differentiating \((I-B_u)^{-1}-I\) and using
\[
A_u^T\Gamma_u^{-1}=(I-B_u)^{-1}A_u^T
\]
gives
\[
-S_u'
=A_u^T\Gamma_u^{-1}\bbeta(u)\Gamma_u^{-1}A_u.
\]
The same expression equals \(D_u^T\bbeta(u)^{-1}D_u\), proving the second identity in \eqref{eq:ld-tweedie}.
\end{proof}

Recall \(\Rbar_u\) from Section~4.1 and set
\[
\mathcal J_u:=\E[h_u(X_u)h_u(X_u)^T].
\]
The unconditional identity
\begin{equation}\label{eq:ld-renormalized-energy}
\Rbar_u
=
A_u^{-1}
\{\Sigma_u-2\Gamma_u+\Gamma_u\mathcal J_u\Gamma_u\}
A_u^{-T}
\end{equation}
follows directly from
\[
m_u(X_u)=A_u^{-1}\{X_u+\Gamma_uh_u(X_u)\}
\]
and Gaussian integration by parts,
\[
\E[X_uh_u(X_u)^T]
=
\E[h_u(X_u)X_u^T]
=-I.
\]
Expanding the second moment gives \eqref{eq:ld-renormalized-energy}. Since
\(A_u^{-T}W_uA_u^{-1}=\Gamma_u^{-1}\bbeta(u)\Gamma_u^{-1}\), multiplication
by \(W_u\) and cyclicity of the trace give
\begin{equation}\label{eq:ld-weighted-renormalized-energy}
\Tr(W_u\Rbar_u)
=\Tr\{\bbeta(u)\mathcal J_u\}
+\Tr\{\Gamma_u^{-1}\bbeta(u)\Gamma_u^{-1}\Sigma_u\}
-2\Tr\{\Gamma_u^{-1}\bbeta(u)\}.
\end{equation}

\subsection{Gaussian-observation bound}

\begin{proof}[Proof of Lemma~\ref{lem:ld-channel}]
This proof is a matrix-valued counterpart of the covering arguments used to
control posterior covariance in low-dimensional DDPM analysis
\cite{huang2026ddpm}. We use only a finite net and Gaussian concentration,
standard tools for metric-entropy bounds \citep[Chapter~5]{wainwright2019high},
and avoid any posterior-covariance evolution identity. Fix \(u>0\) and whiten
the forward observation:
\[
\widetilde X
:=\Gamma_u^{-1/2}X_u
=\widetilde A_uX_0+Z,
\qquad
\widetilde A_u:=\Gamma_u^{-1/2}A_u,
\qquad
Z\sim N(0,I).
\]
Then \(\widetilde A_u^T\widetilde A_u=S_u\), and conditioning on \(\widetilde X\) is equivalent to conditioning on \(X_u\).

Fix \(r>0\), and let \(\{x_1,\ldots,x_M\}\) be an \(r\)-net of \(\cX\)
in the metric \(d_{\bbeta,u}\). Thus, for every \(x\in\cX\), some \(x_j\)
satisfies \(\|\widetilde A_u(x-x_j)\|_2\leq r\).

Define the minimum-distance estimator
\[
\widehat x(\widetilde X)
\in
\operatorname*{argmin}_{1\leq j\leq M}
\|\widetilde X-\widetilde A_ux_j\|^2.
\]
Condition on \(X_0=x\), let \(x_{j_0}\) be a nearest net point, and write
\[
a_j:=\widetilde A_u(x-x_j).
\]
The minimizing property gives
\[
\|a_{\widehat j}\|^2+2\langle Z,a_{\widehat j}\rangle
\leq
\|a_{j_0}\|^2+2\langle Z,a_{j_0}\rangle.
\]
Consequently,
\begin{equation}\label{eq:ld-minimum-distance}
\frac12\|a_{\widehat j}\|^2
\leq
\|a_{j_0}\|^2
+2\langle Z,a_{j_0}\rangle
+\max_{1\leq j\leq M}
\left\{
-2\langle Z,a_j\rangle-\frac12\|a_j\|^2
\right\}.
\end{equation}
For every deterministic vector \(a\),
\[
\E\exp\left[
\frac18
\left\{
-2\langle Z,a\rangle-\frac12\|a\|^2
\right\}
\right]
\leq1.
\]
The log-sum-exp inequality therefore implies
\[
\E
\max_{1\leq j\leq M}
\left\{
-2\langle Z,a_j\rangle-\frac12\|a_j\|^2
\right\}
\leq8\log M.
\]
Taking expectations in \eqref{eq:ld-minimum-distance}, using
\(\E\langle Z,a_{j_0}\rangle=0\), and observing that
\[
\|a_{j_0}\|^2\leq r^2,
\]
we obtain
\[
\E\|\widetilde A_u\{X_0-\widehat x(\widetilde X)\}\|^2
\leq
C\{\log M+r^2\}.
\]

The conditional mean minimizes Bayes risk for every deterministic positive semidefinite quadratic loss. Hence
\begin{equation}\label{eq:ld-finite-net}
\mathcal I_{\mathrm{pm}}(u)
=
\E\|\widetilde A_u\{X_0-m_u(X_u)\}\|^2
\leq
C\{\log M+r^2\}.
\end{equation}
Minimizing over \(r\) proves \eqref{eq:ld-anisotropic-complexity}. Finally, a
Euclidean \(\varepsilon\)-net is an
\(\|S_u\|_2^{1/2}\varepsilon\)-net in \(d_{\bbeta,u}\). Substitution in
\eqref{eq:ld-anisotropic-finite-net} proves
\eqref{eq:ld-finite-net-main}.
\end{proof}

\section{Matrix OU identities}\label{app:technical-identities}

This appendix records the elementary identities used in Section~2. All computations are first justified for smooth positive densities with sufficient decay; the general case follows by applying the identities at positive time after OU smoothing and then using a standard approximation argument.

\subsection{Forward calculus}

Since \(q_u=\gamma r_u\) and \(\nabla\gamma=-x\gamma\), the Fokker--Planck equation for \eqref{eq:intro-forward-SDE} gives
\[
\partial_u q_u
=\frac12\nabla\cdot\{\bbeta(u)(xq_u+\nabla q_u)\}
=\frac12\nabla\cdot\{\bbeta(u)\gamma\nabla r_u\}.
\]
Dividing by \(\gamma\) yields \(\partial_u r_u=\mathcal L_ur_u\). Moreover, Gaussian integration by parts gives, for smooth \(f,g\),
\begin{equation}\label{eq:app-gaussian-ibp}
\int f\mathcal L_ug\,d\pi
=-\frac12\int\langle\nabla f,\bbeta(u)\nabla g\rangle d\pi.
\end{equation}

Set \(\ell_u=\log r_u\), \(\zeta_u=\nabla\ell_u\), \(H_u=\nabla^2\ell_u\), and \(\varphi_u=\zeta_u^T\bbeta(u)\zeta_u\). Direct differentiation gives
\[
\nabla\varphi_u=2H_u\bbeta(u)\zeta_u
\]
and
\[
\Tr\{\bbeta(u)\nabla^2\varphi_u\}
=2\|\bbeta(u)^{1/2}H_u\bbeta(u)^{1/2}\|_F^2
+2\langle\bbeta(u)\zeta_u,\nabla\Tr\{\bbeta(u)H_u\}\rangle.
\]
Combining these identities with
\[
\nabla(\mathcal L_u\ell_u)
=\frac12\nabla\Tr\{\bbeta(u)H_u\}
-\frac12H_u\bbeta(u)x-\frac12\bbeta(u)\zeta_u
\]
gives the Bochner identity
\begin{equation}\label{eq:app-bochner}
\mathcal L_u\varphi_u-2\langle\bbeta(u)\zeta_u,\nabla(\mathcal L_u\ell_u)\rangle
=\|\bbeta(u)^{1/2}H_u\bbeta(u)^{1/2}\|_F^2
+\zeta_u^T\bbeta(u)^2\zeta_u.
\end{equation}

\subsection{Information dissipation}

The relative forward equation, \(\int r_u\,d\pi=1\), and \eqref{eq:app-gaussian-ibp} give
\[
\frac{d}{du}\KL(q_u\|\pi)
=\int (\mathcal L_ur_u)\log r_u\,d\pi
=-\frac12 I_\bbeta(u),
\]
which proves Proposition~\ref{prop:kl-dissipation}.

Since
\[
\partial_u\ell_u=\mathcal L_u\ell_u+\frac12\varphi_u,
\qquad
\partial_u\zeta_u=\nabla(\mathcal L_u\ell_u)+\frac12\nabla\varphi_u,
\]
differentiating \(I_\bbeta(u)=\int r_u\varphi_u\,d\pi\) gives
\[
\frac{d}{du}I_\bbeta(u)
=\int (\mathcal L_ur_u)\varphi_u\,d\pi
+2\int r_u\langle\bbeta(u)\zeta_u,\partial_u\zeta_u\rangle d\pi
+I_{\bbeta'}(u).
\]
Using \eqref{eq:app-gaussian-ibp} and \eqref{eq:app-bochner}, the first two terms equal \(-J_\bbeta(u)-I_{\bbeta^2}(u)\). Hence
\[
\frac{d}{du}I_\bbeta(u)
=I_{\bbeta'}(u)-I_{\bbeta^2}(u)-J_\bbeta(u),
\]
which proves Proposition~\ref{prop:fisher-evolution}.

\section{KL decomposition}\label{app:kl-decomposition-proof}

\begin{proof}[Proof of Theorem~\ref{theo:kl-decomposition}]
For score freezing, set \(a_t:=\bbeta(T-t)\). Let
\(P_{\mathrm{score}}^{q_T}\) and
\(P_{\mathrm{score}}^\pi\) denote the path laws of
\eqref{eq:score-frozen-sampler} initialized from \(q_T\) and \(\pi\),
respectively. The true and discretized processes with initial law \(q_T\) have
the same diffusion covariance, and their drift difference on
\([t_k,t_{k+1})\) is
$$
a_t\left(\nabla\log q_{T-t}(Y_t)-s_\theta(Y_{t_k},u_k)\right).
$$
Girsanov's theorem, justified by stopping-time truncation, gives
\begin{equation}\label{eq:appendix-girsanov}
\KL(Q\|P_{\mathrm{score}}^{q_T})
\leq \frac12\sum_{k=0}^{N-1}\int_{t_k}^{t_{k+1}}
\E_Q\left[
\left\|\nabla\log q_{T-t}(Y_t)-s_\theta(Y_{t_k},u_k)\right\|_{\bbeta(T-t)}^2
\right]dt.
\end{equation}
The endpoint laws of \(Q\) and \(P_{\mathrm{score}}^\pi\) are
\(q_{t_{\min}}\) and \(p_{T-t_{\min}}\). Hence data processing and the entropy
chain rule give
$$
\KL(q_{t_{\min}}\|p_{T-t_{\min}})
\leq
\KL(Q\|P_{\mathrm{score}}^{q_T})+\KL(q_T\|\pi).
$$
Finally, insert and subtract \(\nabla\log q_{u_k}(Y_{t_k})\) in \eqref{eq:appendix-girsanov} and use
$$
\frac12\|x+y\|_A^2\leq\|x\|_A^2+\|y\|_A^2.
$$
The budget \eqref{eq:score-error-budget} bounds the score-approximation contribution by \(\epsilon_{\mathrm{score}}^2\), while the remaining contribution is the exact-score discretization term in Theorem~\ref{theo:kl-decomposition}.

For posterior-mean freezing, let \(P_{\mathrm{pm}}^{q_T}\) and
\(P_{\mathrm{pm}}^\pi\) denote the path laws
of \eqref{eq:ld-sampler} initialized from \(q_T\) and \(\pi\). On the \(k\)th
interval, the drift difference from the exact reverse process is
\[
D_u\{m_u(Y_t)-\widehat m_{u_k}(Y_{t_k})\},
\qquad u=T-t.
\]
Using \eqref{eq:ld-tweedie}, Girsanov's theorem gives
\[
\KL(Q\|P_{\mathrm{pm}}^{q_T})
\leq
\frac12\sum_{k=0}^{N-1}\int_{t_k}^{t_{k+1}}
\E_Q
\|m_u(Y_t)-\widehat m_{u_k}(Y_{t_k})\|_{W_u}^2dt.
\]
Insert and subtract \(m_{u_k}(Y_{t_k})\). Under time reversal,
\((Y_t,Y_{t_k})\) has the law of \((X_u,X_{u_k})\), and
\eqref{eq:ld-markov-identity} gives
\[
\E_Q[m_u(Y_t)-m_{u_k}(Y_{t_k})\mid Y_{t_k}]=0.
\]
The cross term therefore vanishes for the deterministic matrix \(W_u\). After
the change of variables \(u=T-t\), the two remaining terms are exactly
\(\operatorname{Disc}_{\mathrm{pm}}\) and
\(\operatorname{Err}_{\mathrm{pm}}\). Data processing and the entropy chain
rule add \(\KL(q_T\|\pi)\), proving the posterior-mean inequality.
\end{proof}

\section{Schedule and adaptive-grid proofs}
\label{app:adaptive-grid-proof}

\begin{proof}[Proof of Proposition~\ref{prop:forward-score-curvature}]
For any symmetric matrix \(A\), the integration-by-parts calculation in
\eqref{eq:K-information} gives
\begin{equation}\label{eq:app-ordinary-relative-fisher}
\mathscr K_A(u)
=I_A(u)+2\Tr A-\Tr(A\Sigma_u).
\end{equation}
Recall that \(\Sigma_u=\E[X_uX_u^T]\). The forward equation implies
\begin{equation}\label{eq:app-second-moment-evolution}
\Sigma_u'
=\bbeta(u)-\frac12\{\bbeta(u)\Sigma_u+\Sigma_u\bbeta(u)\}.
\end{equation}
Differentiate \eqref{eq:app-ordinary-relative-fisher} at
\(A=\bbeta(u)\), use Proposition~\ref{prop:fisher-evolution}, and then apply
\eqref{eq:app-second-moment-evolution}. This gives
\begin{equation}\label{eq:app-ordinary-fisher-intermediate}
\frac{d}{du}\mathscr K_{\bbeta(u)}(u)
=\mathscr K_{\dot\bbeta(u)}(u)
-\mathscr K_{\bbeta(u)^2}(u)
+\Tr\{\bbeta(u)^2\}-J_\bbeta(u).
\end{equation}
Since \(H_u-I=\nabla^2\log q_u\), another integration by parts gives
\[
\int q_u(x)\{H_u(x)-I\}\,dx
=-\int q_u(x)h_u(x)h_u(x)^Tdx.
\]
Expanding \(J_\bbeta\) from \eqref{eq:Jbeta} around \(H_u-I\) therefore
yields
\[
J_\bbeta(u)
=\widetilde J_{\bbeta(u)}(u)
-2\mathscr K_{\bbeta(u)^2}(u)
+\Tr\{\bbeta(u)^2\}.
\]
Substitution into \eqref{eq:app-ordinary-fisher-intermediate} proves
\eqref{eq:ordinary-fisher-evolution}. Finally,
\(K(s)=\mathscr K_{\bbeta(u)}(u)\) for \(u=T-s\), while the second term in
the score formula of \eqref{eq:adaptive-growth-formulas} is
\(\mathscr K_{\dot\bbeta(u)+\bbeta(u)^2}(u)\). The two terms cancel according
to \eqref{eq:ordinary-fisher-evolution}, leaving
\(g_{\mathrm{sc}}(s)=\widetilde J_{\bbeta(u)}(u)\).
\end{proof}

\begin{proof}[Proof of Corollary~\ref{prop:gaussian-schedule-diagnostic}]
For a Gaussian marginal,
\(\nabla^2\log q_u=-V_u^{-1}\), so
\eqref{eq:gaussian-schedule-growth} follows from
Proposition~\ref{prop:forward-score-curvature}. In two dimensions, diagonalize
\(V_u\) and write
\(\bbeta=R_\theta\operatorname{diag}(b_1,b_2)R_\theta^T\), with
\(b_1\geq b_2\), and let \(v_1\geq v_2\) be the eigenvalues of \(V_u\).
Expanding \eqref{eq:gaussian-schedule-growth} as a quadratic function of
\(\sin^2\theta\) shows that it is minimized at \(\theta=0\). Thus the larger
rate is paired with the larger variance.
\end{proof}

\begin{proof}[Proof of Proposition~\ref{prop:intrinsic-schedule-alignment}]
Because the matrices in \eqref{eq:intrinsic-permuted-schedule} share the fixed
eigenbasis \(Q\),
\[
A_u^\sigma
=Q\operatorname{diag}\left\{
e^{-\Lambda_{\sigma(1)}(u)/2},\ldots,
e^{-\Lambda_{\sigma(d)}(u)/2}
\right\}Q^T.
\]
It follows from \(S_u=A_u^T(I-A_uA_u^T)^{-1}A_u\) that
\begin{equation}\label{eq:app-intrinsic-s-eigenvalues}
S_u^\sigma
=Q\operatorname{diag}\left\{
\frac1{e^{\Lambda_{\sigma(1)}(u)}-1},\ldots,
\frac1{e^{\Lambda_{\sigma(d)}(u)}-1}
\right\}Q^T.
\end{equation}

Suppose a rate curve indexed by \(p>k\) is assigned to a coordinate
\(i\leq k\), while a curve indexed by \(q\leq k\) is assigned to a coordinate
\(\ell>k\). Swap these two assignments. Since
\(\Lambda_q(u)\geq\Lambda_p(u)\) and
\(z\mapsto(e^z-1)^{-1}\) is decreasing, the coefficient of
\((q_i^T(x-y))^2\) in \(d_{\bbeta,u}(x,y)^2\) does not increase. The changed
coefficient in direction \(q_\ell\) is irrelevant because
\(x-y\in\mathcal U\). Repeating this exchange produces a permutation
\(\sigma^\uparrow\) with the required assignment and proves
\eqref{eq:intrinsic-alignment-metric}.

A pointwise smaller metric has no larger covering number at any radius.
Therefore \(\mathsf H_{\bbeta_{\sigma^\uparrow},u}(r)\leq
\mathsf H_{\bbeta_\sigma,u}(r)\) for every \(r>0\); taking the infimum proves
\eqref{eq:intrinsic-alignment-complexity}. The last claim follows by inserting
this comparison into \eqref{eq:ld-main-disc-anisotropic}.
\end{proof}

\begin{proof}[Proof of Proposition~\ref{prop:adaptive-asymptotic}]
Write \(\Delta_i=t_{i+1}-t_i\). On each smooth schedule interval,
\eqref{eq:adaptive-local-growth} gives, locally uniformly in \(s\),
\[
\mathcal C_\star[s,s+\Delta]
=\frac12g_\star(s)\Delta^2+o(\Delta^2)
\qquad (\Delta\downarrow0).
\]
Since \(g_\star\) is bounded above and away from zero,
Definition~\ref{def:adaptive-grid} gives a mesh of order \(N^{-1}\). Apart
from the finitely many grid intervals that meet a schedule breakpoint, the
integral mean-value theorem gives \(\zeta_i\in[t_i,t_{i+1}]\) such that
\[
\sqrt{g_\star(\zeta_i)}\,\Delta_i=\frac{L_\star}{N}.
\]
Piecewise continuity and the local expansion then give
\[
\sum_{i=0}^{N-1}\mathcal C_\star[t_i,t_{i+1}]
=\frac12\sum_{i=0}^{N-1}g_\star(t_i)\Delta_i^2+o(N^{-1})
=\frac{L_\star^2}{2N}+o(N^{-1}),
\]
because the finitely many breakpoint intervals contribute \(O(N^{-2})\).
This proves \eqref{eq:adaptive-grid-optimal-value}.

For any competing grid with mesh \(O(N^{-1})\), Cauchy--Schwarz gives
\[
\sum_{i=0}^{N-1}g_\star(t_i')({\Delta_i'})^2
\geq\frac1N
\left\{\sum_{i=0}^{N-1}
\sqrt{g_\star(t_i')}\,\Delta_i'\right\}^2.
\]
The expression in braces converges to \(L_\star\). Omitting the finitely many
breakpoint intervals changes it by \(O(N^{-1})\), while their costs are
nonnegative. The local expansion therefore proves
\eqref{eq:adaptive-grid-lower-bound}.
\end{proof}

\begin{proof}[Proof of Proposition~\ref{prop:adaptive-dimension}]
We first identify the two local growth rates. Put \(u=T-s\). On each smooth
schedule interval, the transition equation and \(C^1\) regularity give
\[
C_{s,s+r}=r\{\dot\bbeta(u)+\bbeta(u)^2\}+o(r).
\]
Hence
\[
g_{\mathrm{sc}}(s)
=K'(s)+\Tr\!\left[
\{\dot\bbeta(u)+\bbeta(u)^2\}\mathcal H_s\right].
\]
Similarly, differentiability of \(\Rbar_u\) gives
\[
g_{\mathrm{pm}}(s)
=\frac12\Tr\{W_u(-\Rbar_u')\}.
\]

For score freezing, write
\[
\bar K(u):=K(T-u),
\qquad
w(u):=\min\{1,u\},
\qquad
\mathscr K_A(u):=\E[h_u(X_u)^TAh_u(X_u)].
\]
The preceding identity and \(\bbeta'(u)\preceq0\) give
\[
g_{\mathrm{sc}}(T-u)
=
-\bar K'(u)
+\mathscr K_{\bbeta'(u)+\bbeta(u)^2}(u)
\leq
-\bar K'(u)+B_\bbeta\bar K(u).
\]
Lemma~\ref{lem:cumulative-technical-bounds} and integration by parts yield
\[
\begin{aligned}
\int_{t_{\min}}^T
w(u)(-\bar K'(u))\,du
&\leq
Cd\left(1+\log\frac1{t_{\min}}\right),\\
\int_{t_{\min}}^T
w(u)\bar K(u)\,du
&\leq Cd(1+T).
\end{aligned}
\]
Since
\[
\int_{t_{\min}}^T\frac{du}{w(u)}
=
T-1+\log\frac1{t_{\min}},
\]
weighted Cauchy--Schwarz gives
\[
L_{\mathrm{sc}}^2
\leq
Cd\left(T+\log\frac1{t_{\min}}\right)^2.
\]

For posterior-mean freezing, apply the generic lower bound to any
rate-regular reference-grid sequence. Theorem~\ref{theo:ld-cumulative} gives
\[
\frac12L_{\mathrm{pm}}^2
\leq
\liminf_{N\to\infty}
N\operatorname{Disc}_{\mathrm{pm}}(\mathcal T_N)
\leq
Ck\log k
\left(T+\log\frac1{t_{\min}}\right)^2.
\]
\end{proof}
\bibliography{references}

\end{document}